\documentclass[11pt,oneside,reqno]{amsart}

\usepackage{mathrsfs} 
\usepackage{enumitem}
\usepackage{amssymb} 
\theoremstyle{plain}
\newtheorem{Theorem}{Theorem}

\newtheorem{theorem}[Theorem]{Theorem}
\newtheorem{proposition}[Theorem]{Proposition}

\newtheorem{corollary}[Theorem]{Corollary}

\newtheorem{lemma}[Theorem]{Lemma}

\theoremstyle{definition}

\newtheorem{example}[Theorem]{Example}
\newtheorem{definition}[Theorem]{Definition}
\newtheorem{remark}[Theorem]{Remark}

\theoremstyle{remark}

\begin{document}

%
%
%
%
%
%
%
%
%

\title{Quasi-linear maps and image transformations}

\author{S. V. Butler}

\address{Department of Mathematics\\
University of California, Santa Barbara\\
Isla Vista, CA 93117\\
USA}

\email{svetbutler@gmail.com}

\begin{abstract}
Conic quasi-linear maps are nonlinear operators from $C_0(X)$ to a normed linear space $E$ 
which preserve nonnegative linear combinations on positive cones generated by single functions; 
quasi-linear maps are linear on singly generated subalgebras.
While nonlinear, a quasi-linear map is bounded iff it is continuous.  $E = \mathbb{R}$ gives quasi-integrals, which
correspond to (deficient) topological measures - nonsubadditive set functions generalizing measures. 
Like image measures $\mu \circ u^{-1}$, (d-) image transformations move
(deficient) topological measures from one space to another, generalizing $u^{-1}$.
We give criteria for a (d-) image transformation to be  $u^{-1}$ for some proper continuous function.
We study the interrelationships between (conic) quasi-linear maps, quasi-integrals, 
(deficient) topological measures and (d-) image transformations when $E = C_0(Y), X, Y$ are locally compact.
(Conic) quasi-homomorphisms behave like homomorphisms on singly generated subalgebras or cones. 
We show that (conic) quasi-homomorphisms are in 1-1 correspondence with (d-) image transformations 
and with certain continuous proper functions. 
We give criteria for a (conic) quasi-linear map to be a (conic) quasi-homomorphism, 
and for the latter to be an algebra homomorphism. Any conic quasi-linear map or bounded quasi-linear map 
is a composition of an algebra homomorphism with the basic quasi-linear map, 
and we give criteria for the latter to be linear.
We study the adjoints of (d-) image transformations and (conic) quasi-linear maps; for 
(conic) quasi-homomorphisms they give Markov-Feller operators with nonlinear duals.
\end{abstract}

\subjclass{47H99, 47H07, 47H04, 28A33}

\keywords{Quasi-linear maps and their adjoints,  (d-) image transformations, quasi-homomorphisms, deficient topological measures, p-conic quasi-linear functionals, 
k-proper functions}

\date{Jan. 17, 2025}

\maketitle






\section{Introduction}

In this paper we introduce the concepts of conic quasi-linear maps and d- image transformations. We also generalize ideas of 
quasi-linear maps and image transformations and all major results about them to locally compact spaces. 
We study the intricate interplay between (conic) quasi-linear maps, quasi-integrals, 
(deficient) topological measures and (d-) image transformations.

Conic quasi-linear maps are nonlinear operators from $C_0(X)$ to a normed linear space $E$ 
which preserve nonnegative linear combinations on each positive cone generated by a single function. 
Quasi-linear maps are linear on each subalgebra  generated by a single function.  
Being nonlinear, these operators nevertheless have some important properties of linear operators. 
For example, a quasi-linear map is bounded iff it is continuous. An important subclass 
of quasi-linear maps are quasi-homomorphisms and conic quasi-homomorphisms, which behave like homomorphisms on singly 
generated subalgebras or, more generally, on singly generated cones.  

When $E = \mathbb R$, we obtain quasi-linear functionals and p-conic quasi-linear functionals. 
Such nonlinear functionals (also called quasi-integrals) correspond to topological measures and deficient topological measures.
These set functions generalize measures, and their definitions are close to that of a regular Borel measure. 
However, (deficient) topological measures are defined on open and closed subsets of a topological space, which leads to some striking differences from 
Borel measures. For instance,  (deficient) topological measures are not subadditive.  
The theory of quasi-linear functionals and topological measures, introduced by J. F. Aarnes \cite{Aarnes:TheFirstPaper},
has origins in mathematical axiomatization and interpretations of quantum physics.
Later this theory proved to be influential in symplectic geometry, 
playing an important role in function theory on symplectic manifolds (\cite{PoltRosenBook}). 
The theory of quasi-linear functionals and topological measures is also connected to other areas in mathematics, 
including theory of Choquet integrals, fractals, and probability and statistics.
A full bibiliography on the origins of the theory of quasi-linear functionals 
and topological measures and its connections to other areas of mathematics would be very long,  
but \cite{AarnesJohansenRustad},  \cite{Butler:QLFLC},  \cite{Butler:WkConv}, \cite{AlfImTrans}, \cite{AlfMedian}, and \cite{Svistula:Choquet} 
give more information on these matters and provide many additional references.
In Section \ref{Prelim} we give some basic information and provide some examples.

In this paper we are often interested in (conic) quasi-linear maps from $C_0(X)$ to $C_0(Y)$, when $X,Y$ are locally compact. 
Such maps allow us to move quasi-integrals from one space to another.
We give several criteria for a (conic) quasi-linear map to be a (conic) quasi-homomorphism. 
The composition of quasi-linear maps need not be a quasi-linear map, but the composition of (conic) quasi-homomorphisms is a 
(conic) quasi-homomorphism. 
The adjoint of a (conic) quasi-homorphism (and some other quasi-linear maps) is a continuous 
operator between spaces of (deficient) topological measures equipped with the weak topology.  
If $ \theta$ is a (conic) quasi-linear map, then for any $ y \in Y$ we obtain a quasi-integral $\theta(f)(y), f \in C_0(X)$. 
Conversely, from quasi-integrals we can obtain a (conic) quasi-linear map with this property.  
This is possible because (conic) quasi-linear maps from $C_0(X)$ to $C_0(Y)$ correspond to  certain proper functions from
$Y$ to the space of (deficient) topological measures on $X$. We also study the basic quasi-linear map which maps $f \in C_0(X)$ to 
a continuous function $ \tilde f$ on the set of (deficient) topological measures, where $ \tilde f(\mu)$ is the quasi-integral of $f$ with respect to $ \mu$.    
These and other results are in Section \ref{SectQLM}.

Image transformations and d-image transformations generalize the idea of  $u^{-1}$ in image measures $\mu \circ u^{-1}$. 
They move subsets of one space to subsets of another space so that (deficient) 
topological measures on the second space produce (deficient) topological measures on the first space.  
We give criteria for a (d-) image transformation to be the inverse of some proper continuous function.
We study properties of a (d-) image transformation,  define its adjoint and show that it is continuous. 
We also show that there is a 1-1 correspondence between (d-) image transformations from $X$ to $Y$ and certain continuous proper functions from 
$Y$ to $\{0,1\}$-valued (deficient) topological measures on $X$. We give many examples, showing, for instance, 
that some image transformations can annihilate sets of certain size, or that the uniqueness of a 
Haar measure on a locally compact topological group
fails in the more general context of topological measures.  These results are in Section \ref{SectIT}.   

We show that there is a 1-1 correspondence between (d-) image transformations and (conic) quasi-homomorphisms. 
This allows us to study many properties of both objects and obtain important results. 
We give several criteria for a (conic) quasi-homomorphism to be an algebra homomorphism.  
We further study the basic (conic) quasi-homorphism and give many equivalent conditions for it to be linear. 
We show that any conic quasi-linear map or bounded quasi-linear map can be represented as a composition of an algebra homomorphism with 
the basic quasi-linear map. We give several criteria for a (deficient) topological measure to be in the closed convex hull of $\{0,1\}$-valued 
(deficient) topological measures.  The adjoint of a (d-) image transformation 
is a Markov-Feller operator, whose dual is a quasi-linear map, a nonlinear operator.
These and other results are in Section \ref{ITqhCorr}.

In this paper we introduced the concepts of d-image transformations, k-proper functions, and conic quasi-linear maps. 
Image transformations and quasi-homomorphisms were studied before in the context (or close to the context) of compact spaces. 
We generalize all major results and examples from previous works. 

\section{Preliminaries} \label{Prelim}

In this paper all spaces are assumed to be Hausdorff. We may abbreviate and write LC for "locally compact".  
By $C(X)$ we denote the set of all real-valued continuous functions on $X$ with the extended uniform norm, 
by $C_b(X)$ the normed space of bounded continuous functions on $X$,   
by $C_0(X)$ the set of continuous functions on $X$ vanishing at infinity, 
by $C_c(X)$ the set of continuous functions with compact support, and  by 
$C_0^+(X)$, $C_b^+(X)$, $C_c^+(X)$, $C^+(X)$ the collections of all nonnegative functions from $C_0(X)$, $C_b(X)$, $C_c(X)$, $C(X)$, respectively.
When we consider maps into extended real numbers we assume that any such map is not identically $\infty$. 
We denote by $\overline E$ the closure of a set $E$, and by $ \bigsqcup$ a union of disjoint sets.
Notation $K_t \searrow K$ means that a decreasing net of sets ($t < s \Longrightarrow K_s \subseteq K_t $) decreases to 
$K = \bigcap_{t \in T} K_t$. Similarly, $U_t \nearrow U$ stands for an increasing to $U = \bigcup_{t \in T} U_t$ net of sets.
We denote by $id$ the identity function $id(x) = x$, 
by $1_K$ the characteristic function of a set $K$,  by $ \delta_x$ the point mass at $x$, and by $P_e(X)$ the set of all point masses on $X$. 
By $ supp \,  f $ we mean $ \overline{ \{x: f(x) \neq 0 \} }$, and  
$Coz f = \{x: f(x) \neq 0\}$.
Several collections of sets are used often.   They include:
$\mathscr{O}(X)$;
$\mathscr{C}(X)$; and
$\mathscr{K}(X)$-- 
the collection of open subsets of $X$;  the collection of closed subsets of   $X $;
and the collection of compact subsets of   $X $, respectively.

\begin{definition} \label{MDe2}
Let $X$ be a  topological space and $\nu$ be a set function on a family $\mathcal{E}$ of subsets of $X$ that 
contains $\mathscr{O}(X) \cup \mathscr{C}(X)$
with values in $[0, \infty]$. 
We say that $\nu$ is
 simple if it only assumes  values $0$ and $1$;
 finite if $ \nu(X) < \infty$; 
compact-finite  if $\nu(K) < \infty$ for every compact $K$; 
 $\tau-$ smooth on compact sets if for every decreasing net
$K_{\alpha} \searrow K, K_\alpha, K \in \mathscr{K}(X)$ one has $\nu(K_{\alpha}) \rightarrow \nu(K)$;
$\tau-$ smooth on open sets if for every increasing net
$U_{\alpha} \nearrow U, U_\alpha, U \in \mathscr{O}(X)$ one has $\nu(U_{\alpha})  \rightarrow \nu(U)$.  
\end{definition}

\begin{definition}
A measure on $X$ is a countably additive set function on a $\sigma$-algebra of subsets of $X$ with values in $[0, \infty]$.  
A Borel measure on $X$ is a measure on the Borel $\sigma$-algebra on $X$.  
A Radon measure  $m$  on $X$ is  a compact-finite Borel measure that is outer regular on all Borel sets and inner regular on all open sets, i.e.
$ m(E) = \inf \{ m(U): E \subseteq U, U \text{  is open} \} $ for every Borel set $E$, and 
$m(U) = \sup \{  m(K): K \subseteq U, K  \text{  is compact} \}$ for every open set $U$. 
A Borel measure is regular if it is outer regular and inner compact regular on all Borel sets.   
\end{definition}

For the following fact see, for example,  \cite[Ch. XI, 6.2]{Dugundji} and \cite[L. 2.5]{Butler:TMLCconstr}.
\begin{lemma} \label{easyLeLC}
Let $K \subseteq U, \ K \in \mathscr{K}(X),  \ U \in \mathscr{O}(X)$ in a LC space $X$.
Then there exists a set  $V \in \mathscr{O}(X)$ such that $ C = \overline V$ is compact and  
$ K \subseteq V \subseteq \overline V \subseteq U. $
If $X$ is also locally connected, and either $K$ or $U$ is connected, then $V$ and $C$ can be chosen to be connected. 
\end{lemma}

Some of the major objects in this paper are (deficient) topological measures and (p-conic) quasi-linear functionals. 

\begin{definition}\label{TMLC}
A topological measure on $X$ is a set function
$\mu:  \mathscr{C}(X) \cup \mathscr{O}(X)  \longrightarrow  [0,\infty]$ satisfying the following conditions:
\begin{enumerate}[label=(TM\arabic*),ref=(TM\arabic*)]
\item \label{TM1} 
if $A,B, A \sqcup B \in \mathscr{K}(X) \cup \mathscr{O}(X) $ then
$
\mu(A\sqcup B)=\mu(A)+\mu(B);
$
\item \label{TM2}  
$
\mu(U)=\sup\{\mu(K):K \in \mathscr{K}(X), \  K \subseteq U\}
$ for $U\in\mathscr{O}(X)$;
\item \label{TM3}
$
\mu(F)=\inf\{\mu(U):U \in \mathscr{O}(X), \ F \subseteq U\}
$ for  $F \in \mathscr{C}(X)$.
\end{enumerate}
If in \ref{TM1}  $A,B \in \mathscr{K}(X)$ then $\mu$ is called a deficient topological measure. 

For  a deficient topological measure  $\mu$ we define $ \| \mu \| = \mu(X) = \sup \{ \mu(K): K \in \mathscr{K}(X) \}$.
\end{definition} 

\noindent
Clearly, for a closed set $F$, $ \nu(F) = \infty$ iff $ \nu(U) = \infty$ for every open set $U$ containing $F$.
If two (deficient) topological measures agree on compact sets (or on open sets) then they coincide. If for (deficient) topological measures 
$\mu(A) \le \nu(A)$ for all open (or all compact) sets  $A$, then $ \mu \le \nu$.

\begin{remark} \label{tausm}
One may consult  \cite{Butler:DTMLC} for various properties of deficient topological measures on LC spaces.
A deficient topological measure is monotone, so a real-valued deficient topological measure on a LC space is a regular content. 
A deficient topological measure $ \nu$ is $\tau$-smooth on compact sets and $\tau$-smooth on open sets, 
and, in particular, is additive on open sets.
If $ F \in \mathscr{C}(X)$ and $C \in \mathscr{K}(X)$ are disjoint, then $ \nu(F) + \nu(C) = \nu ( F \sqcup C)$.
A deficient topological measure $ \nu$ is also superadditive, i.e. 
if $ \bigsqcup_{t \in T} A_t \subseteq A, $  where $A_t, A \in \mathscr{O}(X) \cup \mathscr{C}(X)$,  
and at most one of the closed sets (if there are any) is not compact, then 
$\nu(A) \ge \sum_{t \in T } \nu(A_t)$. 
Note that unlike measures, simple (deficient) topological measures are not all the extreme (deficient) topological measures  
(see, for example, \cite{QfunctionsEtm}).
\end{remark}

The following criteria for a deficient topological measure to be a measure is from \cite[Sect. 4]{Butler:DTMLC}. 

\begin{theorem} \label{subaddit}
Let $\mu$ be a deficient topological measure on a LC space $X$. 
TFAE: 
\begin{itemize}
\item[(a)]
If $C, K$ are compact subsets of $X$, then $\mu(C \cup K ) \le \mu(C) + \mu(K)$.
\item[(b)]
If $U, V$ are open subsets of $X$,  then $\mu(U \cup V) \le \mu(U) + \mu(V)$.
\item[(c)]
$\mu$ admits a unique extension to an inner regular on open sets, outer regular Borel measure 
$m$ on the Borel $\sigma$-algebra of subsets of $X$. 
$m$ is a Radon measure iff $\mu$ is compact-finite. 
If $\mu$ is finite then $m$ is an outer regular and inner closed regular Borel measure.
\end{itemize}
\end{theorem}

\begin{remark} \label{dtmisptm}
One consequence of superadditivity is this:
if $ \nu$ is a simple deficient topological measure on $X$ and $A$ is a closed or open set with $ \nu(A) =1$, then $ \nu(X \setminus A) = 0$.
If $ \nu(\{x\}) = 1$ for some $x$, then $ \nu(X \setminus \{x \}) = 0$. We see that $\nu$ is subadditive on open sets,  
and by Theorem \ref{subaddit} $\nu$ is the pointmass $ \delta_x$.
\end{remark}

\noindent
Let $TM(X)$ and $DTM(X)$ denote, respectively, the sets of all topological measures and all deficient topological measures on $X$. 

\begin{remark} \label{Vloz}
Let $X$ be LC, and let $ \mathscr{M}$  be the collection of all Borel measures on $X$ that are inner regular on open sets and outer regular 
on Borel sets.  $  \mathscr{M}$ includes regular Borel measures and Radon measures. 
Denote by $M(X)$ the restrictions to $\mathscr{O}(X) \cup \mathscr{C}(X)$ of measures from $ \mathscr{M}$.
Then
\begin{align} \label{incluMTD}
 M(X) \subsetneqq  TM(X)  \subsetneqq  DTM(X).
\end{align}
The inclusions follow from the definitions. 
Information on proper inclusion in (\ref{incluMTD}) 
and various examples are in numerous papers, including
\cite{Aarnes:TheFirstPaper}, \cite{AarnesRustad},  \cite{QfunctionsEtm}, \cite{OrjanAlf:CostrPropQlf},   \cite {Svistula:Signed},  \cite{Svistula:DTM},
\cite{Butler:TechniqLC}, \cite[Sect. 4, Sect. 5]{Butler:DTMLC}, and \cite[Sect. 15]{Butler:TMLCconstr}.
We also give some examples at the end of this section. 
\end{remark}

\begin{definition} \label{cqlf}
We call a  functional $\rho$ on $C_0(X)$  with values in $[ -\infty, \infty]$ (assuming at most one of $\infty, - \infty$) 
and $| \rho(0) | < \infty$ a p-conic quasi-linear functional if 
\begin{enumerate}[leftmargin=0.35in, label=(p\arabic*),ref=(p\arabic*)]
\item
$f\, g=0, \, f, g  \ge 0$ implies $ \rho(f+ g) = \rho(f) + \rho(g)$.
\item
 $0 \le g \le f$ implies $\rho(g) \le \rho(f)$.
\item
For each $f$, if $g,h \in A^+(f), \ a,b \ge 0$ then $\rho(a g + bh) = a \rho(g) + b \rho(h)$.
Here $ A^+(f) = \{ \phi \circ f: \ \phi \in C(\overline{f(X)}), \phi  \mbox{   is nondecreasing}\} $, (with $ \phi(0) = 0 $ 
if $X$ is noncompact) is a cone generated by $f$.
\end{enumerate}

\noindent
A map $\rho:C_0(X) \longrightarrow \mathbb{R}$
is a quasi-linear functional (or a positive quasi-linear functional) if 
\begin{enumerate}[leftmargin=0.25in, label=(\alph*),ref=(\alph*)]
\item \label{QIpositLC}
$ f \ge 0 \Longrightarrow \rho(f)  \ge 0.$
\item \label{QIlinLC}
For each  $f $,   if $g,h \in B(f),  \ a,b \in \mathbb{R}$ then $\rho(a g + bh) = a \rho(g) + b \rho(h)$.
Here  $B(f) =  \{ \phi \circ f : \,  \phi \in C(\overline{f(X)}) \} $  (with $ \phi(0) = 0 $ if $X$ is noncompact) is a subalgebra generated by $f$. 
\end{enumerate}
\end{definition}

\noindent
Note that $B(f)$ is the smallest closed subalgebra of $C_0(X)$ containing $f$ (resp, $f$ and constants if $X$ is compact), see \cite[L. 1.4]{Butler:QLFLC}. 

For a functional $\rho$ on $C_0(X)$ we consider $\| \rho \| =  \sup \{ | \rho(f) | : \  \| f \| \le 1 \} $ and we say $\rho$ is bounded if $\| \rho \| < \infty$. 

\begin{remark} \label{RemBRT}
Let $X$ be a LC space. 
There is an order-preserving bijection between finite deficient topological measures and bounded p-conic quasi-linear functionals. 
There is an order-preserving isomorphism between finite 
topological measures on $X$ and quasi-linear functionals  
of finite norm, and $\mu$ is a measure iff the 
corresponding functional is linear (see \cite[Th. 8.7]{Butler:ReprDTM}, \cite[Th. 3.9]{Alf:ReprTh}, \cite[Th. 15]{Svistula:DTM}, 
and \cite[Th. 3.9]{Butler:QLFLC}).
We outline the correspondence.
\begin{enumerate} [leftmargin=0.2in, label=(\Roman*),ref=(\Roman*)] 
\item \label{prt1}
Given a finite deficient  topological measure $\mu$ on $X$ and $f \in C_b(X)$, define functions on $\mathbb{R}$:
$$ R_1 (t) = R_{1, \mu, f} (t) =  \mu(f^{-1} ((t, \infty) )), \ \ \ \ \   R_2 (t) =  R_{2,  \mu, f} (t) =\mu(f^{-1} ([t, \infty) )). $$
Let $r=r_{f, \mu} = r_f$ be the Lebesque-Stieltjes measure associated with $-R_1$, a regular Borel measure on $ \mathbb{R}$. 
The $ supp \ r \subseteq \overline{f(X)}$.
We define a functional on $C_b(X)$ (in particular, on  $C_0(X)$):
\begin{align} \label{rfform}
\mathcal{R} (f) & = \int _{\mathbb{R}}  id \, dr = \int_{[a,b]} id \, dr  =   \int_a^b R_1 (t) dt + a \mu(X)  =  \int_a^b R_2 (t) dt + a \mu(X), 
\end{align}
where $[a,b]$ is any interval containing $f(X)$.
If $f(X) \subseteq [0,b]$ we have:
\begin{align} \label{rfformp}
 \mathcal{R} (f) = \int_{[0,b]}  id \, dr  =   \int_0^b R_1 (t) dt =   \int_0^b R_2 (t) dt.
\end{align}
We say $\mathcal{R}$  is a quasi-integral (with respect to $ \mu$) and write:
\begin{align*} 
\int_X f \, d\mu = \mathcal{R}(f) = \mathcal{R}_{\mu} (f) =  \int _{\mathbb{R}}  id \, dr.
\end{align*}
If $\mu$ is a topological measure we may write $m = m_{f, \mu} = m_f$ instead of $r$.
\item  \label{RHOsvva}
Functional $\mathcal{R} $ is nonlinear. 
By \cite[L. 7.7,  Th. 7.10, L. 3.6, L. 7.12, Th. 8.7]{Butler:ReprDTM}  and  \cite[Cor. 4.10]{Butler:QLFLC} we have:
\begin{enumerate}[leftmargin=*]
\item
$\mathcal{R} (f) $ is positive-homogeneous, i.e. $\mathcal{R} (cf)  = c \mathcal{R} (f) $ for $c \ge 0$, $ f \in C_b(X)$. 
\item
$\mathcal{R} (0) =0$. 
\item 
$\mathcal{R}$ is monotone, i.e. if $ f \le g$  then $\mathcal{R} (f) \le \mathcal{R} (g)$ for $f, g \in C_b(X)$.
\item
$ \mu(X)  \cdot \inf_{x \in X} f(x)  \le \mathcal{R}(f)  \le \mu(X) \cdot \sup _{x \in X} f(x) $ for $f \in C_b(X)$. 
In particular, 
\begin{align} \label{estin}
  | \int_X f \, d\mu | \le \|f \| \mu(X).
\end{align}
\item
(orthogonal additivity) If $f g = 0$, where  $f, g \in C_b^+(X)$,  then $\mathcal{R} (f+g) = \mathcal{R} (f) + \mathcal{R} (g)$; \\
if $f g = 0 $, where $f \ge 0, g \le 0$, $f, g \in C_0(X)$, then $\mathcal{R} (f+g) = \mathcal{R} (f) + \mathcal{R} (g)$.
\item 
If $ \mu$ is a finite topological measure,  $f, g \in C_0(X)$ then
\begin{align} \label{intsplit}
 \int_X f \, d\mu = \int_X f^+ \, d\mu - \int_X f^- \, d\mu, \ \  \ \ \ \ \ 
|\int_X f \, d\mu - \int_X g \, d\mu | \le  2 \| f - g \| \, \mu(X).
\end{align}

\end{enumerate}
If $X$ is compact and  $\mu$ is a deficient topological measure, then 
$ |  \int_X f \, d\mu - \int_X g \, d\mu | \le \mu(X)  \| f - g \|. $
If $ \mu$ is a topological measure, the functional $ \mathcal{R}$ is a quasi-linear functional. 
If $ \mu$ is a deficient topological measure, the functional $ \mathcal{R}$ is a p-conic quasi-linear functional (which is enough to consider on $C_0^+(X)$).

\item \label{mrDTM} 
A functional $\rho$ with values in $[ -\infty, \infty]$ (assuming at most one of $\infty, - \infty$) and $| \rho(0) | < \infty$ 
is called a d-functional if   
on nonnegative functions it is positive-homogeneous, monotone, and orthogonally additive, i.e. for $f, g \in D(\rho)$ (the domain of $ \rho$) we have: 
(d1) $f \ge 0, \ a > 0  \Longrightarrow  \rho (a f) = a \rho(f)$; 
(d2) $0 \le  g \le f \Longrightarrow  \rho(g) \le \rho(f) $;
(d3) $f \cdot g = 0, f,g \ge 0  \Longrightarrow  \rho(f + g) = \rho(f) + \rho(g)$. 

Let  $\rho$ be a d-functional with $  C_c^+(X) \subseteq D(\rho) \subseteq C_b(X)$. 
In particular, we may take a bounded (p-conic) quasi-linear functional $ \mathcal{R}$ on $  C_0^+(X)$. 
The corresponding
deficient topological measure $ \mu = \mu_{\rho}$ is given as follows: 

If $U$ is open, $ \mu_{\rho}(U) = \sup\{ \rho(f): \,  f \in C_c(X), 0\le f \le 1,  supp \, f\subseteq U  \}$,

if $F$ is closed, $ \mu_{\rho}(F) = \inf \{ \mu_{\rho}(U): \,  F \subseteq U,  U \in \mathscr{O}(X) \}$, 

if $K$ is compact, $ \mu_{\rho}(K) = \inf \{ \rho(g): \,   g \in C_c(X), g \ge 1_K \}  
= \inf \{ \rho(g): \   g \in C_c(X), 1_K \le g \le 1 \}. $
\end{enumerate}
If $\| \rho\| < \infty$ then the corresponding deficient topological measure is finite.  
If given a finite deficient topological measure $\mu$, we obtain $ \mathcal R$, and then $\mu_{ \mathcal R}$, then $ \mu = \mu_{ \mathcal R}$.
A bounded (p-conic) quasi-linear functional $ \mathcal{R}$  has the form given in part \ref{prt1} with $ \mu = \mu_{\mathcal{R}}$. 

For finite deficient topological measures $ \mu, \nu$
\begin{align} \label{EqByInt}
\int f \, d\mu = \int f \, d\nu  \mbox{    for all    }   f \in C_c^+(X)  \ \ \  \Longrightarrow \ \ \  \mu = \nu.
\end{align}

For a finite deficient topological measure $ \mu$ and open $U$ we also have:
\begin{align}  \label{muUprosche}
  \mu(U) =  \sup\{ \int f \, d\mu: \  f \in C_c(X), 0 \le f \le 1_U  \}, 
\end{align} 
since for such $f$ by \cite[Th. 49, Th. 36(b4)]{Butler:Integration}  $\rho(f) = \int_X f \, d\mu = \int_{Coz(f)} f \, d\mu \le \mu(U)$.
\end{remark}

We would like to give some examples. 

\begin{definition}
A set $A \subseteq X$ is called bounded if $\overline A$ is compact. 
If $X$ is LC, noncompact, a set $A$ is solid if $A$ is  connected, and $X \setminus A$ has only unbounded connected components.
If $X$ is compact, a set $A$ is solid if $A$ and $X \setminus A$ are connected.
\end{definition}

Let  $\mathscr{K}_{s}(X)$ and $\mathscr{O}_s^*(X)$ denote the collections of compact solid set and bounded open solid sets in $X$, respectively.  
Let  $ \mathscr{A}_{s}^{*}(X) = \mathscr{K}_{s}(X) \cup \mathscr{O}_s^*(X)$.
 
\begin{definition} \label{DeSSFLC}
A function $ \lambda: \mathscr{A}_{s}^{*}(X) \rightarrow [0, \infty) $ is a solid-set function on $X$ if
\begin{enumerate}[label=(s\arabic*),ref=(s\arabic*)]
\item \label{superadd}
$ \sum_{i=1}^n \lambda(C_i) \le \lambda(C)$ whenever $\bigsqcup_{i=1}^n C_i \subseteq C,  \ \  C, C_i \in \mathscr{K}_{s}(X)$; 
\item \label{regul}
$ \lambda(U) = \sup \{ \lambda(K): \ K \subseteq U , \ K \in \mathscr{K}_{s}(X) \}$ for $U \in \mathscr{O}_{s}^{*}(X)$; 
\item \label{regulo}
$ \lambda(K) = \inf \{ \lambda(U) : \  K \subseteq U, \ U \in \mathscr{O}_{s}^{*}(X) \}$ for $ K  \in \mathscr{K}_{s}(X)$; 
\item  \label{solidparti}
$ \lambda(A) = \sum_{i=1}^n \lambda (A_i)$ whenever $A = \bigsqcup_{i=1}^n A_i, \ \ A , A_i  \in \mathscr{A}_{s}^{*}(X)$.
\end{enumerate}
\end{definition}

\begin{remark}
Many examples of topological measures that are not measures are obtained in the following way. Define a solid-set function on 
a LC connected, locally connected space. A solid set function extends to a unique topological measure. 
See \cite[Def. 2.3, Th. 5.1]{Aarnes:ConstructionPaper}, \cite[Def. 6.1, Th. 10.7]{Butler:TMLCconstr}.
This method is very convenient when $X$ is a compact connected, locally connected space with genus 0 
(i.e. $X$ can not be a disjoint union of more than two nonempty solid sets), and such a space is called a q-space.
For a q-space, in part \ref{solidparti} of Definition \ref{DeSSFLC}  (typically, the hardest to verify) one only needs to check that
$\lambda(X) = \lambda(A) + \lambda(X \setminus A)$ for a solid set $A$. For a noncompact LC space whose one-point compactification
has genus 0 this method is even simpler, for  part \ref{solidparti} of Definition \ref{DeSSFLC} holds trivially. 
For more information about solid sets, solid-set functions, equivalent definitions of a solid-set function, genus,
and more references see \cite[Rem. 6.3, Sect. 11, Sect. 12, L. 15.2]{Butler:TMLCconstr}.
\end{remark}

\begin{example} \label{ExDan2pt}
Suppose that  $ \lambda$ is the Lebesgue measure on $X = \mathbb{R}^2$,  and the set $P$ consists of points
$p_1 = (0,0)$ and $p_2 = (2,0)$.
For each bounded open solid or compact solid set $A$ let $ \nu(A) = 0$ if $A \cap P = \emptyset$,   
$ \nu(A) = \lambda(A) $ if $A$ contains one point from $P$, and 
$ \nu(A) = 2 \lambda(X)$ if $A$ contains both points from $P$.
Then $\nu$ is a solid-set function (see \cite[Ex. 15.5]{Butler:TMLCconstr}), and  $\nu$ extends to a unique topological measure on $X$. 
Let $K_i$ be the closed ball of radius $1$ centered at $p_i$ for $i=1,2$. Then 
$K_1, K_2$ and $ C= K_1 \cup K_2$ are compact solid sets, $\nu(K_1) = \nu(K_2) = \pi, \,  \nu(C) = 4 \pi$. Since 
$\nu$ is not subadditive, it is not a measure.  The quasi-linear functional corresponding to $ \nu$ is not linear. 
\end{example}

\begin{example} \label{nvssf}
Let  $X = \mathbb{R}^2$ or a square, $n$ be a natural number, and let $P$ be a set of distinct $2n+1$ points.
For each $A  \in \mathscr{A}_{s}^{*}(X)$ let $ \nu(A) = i/n$ if $ A$ contains  $2i$ or $2i+1$ points from $P$.
Then $ \nu$ is a solid-set function, and it extends to a unique topological measure on $X$ 
that assumes values $0, 1/n, \ldots, 1$. 
See  \cite[Ex. 2.1]{Aarnes:Pure},  \cite[Ex. 2.5]{AarnesRustad}, \cite[Ex. 4.14, 4.15]{QfunctionsEtm}, 
and \cite[Ex. 15.9]{Butler:TMLCconstr}.
This topological measure is not subadditive. 
(When $X$ is the square and $n=3$, it is easy to represent $X = A_1 \cup A_2 \cup A_3$,
where each $A_i$  is a compact solid set containing one point from $P$. Then $\nu(A_i) =0$ for $i=1,2,3$, while $\nu(X) = 1$.) 
Since $\nu$ is not subbadditive, it is not a measure, and the corresponding quasi-linear functional  $\rho$ is not linear. 
If we take $P =\{p, p, t \} $ or $ \{ p, p, p \}$ then the resulting topological measure is  
the point mass at $p$.
In \cite[Ex. 4.13]{Butler:QLFLC} for $n=5$ we show that there are $f,g \ge 0$ such that $ \rho(f+g) \neq \rho(f) + \rho(g)$. 
If $X$ is LC, noncompact, for the functional $\rho$ we consider
a new functional $ \rho_g$ defined by $\rho_g(f) = \rho(gf)$, where $g \ge 0$. 
The functional $\rho_g$ corresponds to a deficient topological measure 
obtained by integrating $g$ over closed and open sets with respect to $\nu$. We can choose $ g \ge 0$ or $ g >0$ so that  
$\rho_g$ is no longer linear on singly generated subalgebras, but only linear on singly generated cones. 
See \cite[Ex. 35, Th. 43]{Butler:Integration} for details.
\end{example}

\begin{example}[Aarnes circle topological measure] \label{Aatm}
Let $X$ be the unit disk in $ \mathbb{R}^2$ and $B$ be the boundary of $X.$
Fix a point $p$ in the interior of the circle.
Define $\mu $ on solid sets as follows:
$\mu (A) = 1$ if i) $B \subset A$ or
ii) $ p \in A $ and $A \cap B \ne  \emptyset$.
Otherwise $ \mu(A) = 0 $.
Then $ \mu $ is a solid-set function, and it extends to a simple topological measure on $X$.
Let $A_1$ be a closed solid set which is an arc that is a proper subset of $B$, 
$A_2$ be a closed solid set that is the closure of $B \setminus A_1$, 
and $A_3 = X \setminus B$ be
an open solid subset of $X$. Then 
$X =  A_1 \cup A_2 \cup A_3, \  \mu(X) = 1 $, but 
$ \  \mu (A_1) + \mu(A_2) + \mu(A_3) = 0$.
Since $ \mu$ is not subadditive,  it is not a measure, i.e. $\mu$ is simple, but not a point mass. 
When $X = \mathbb{R}^2$, given a closed ball $B(p, \epsilon)$ we may define a set function $\mu_{p, \epsilon}$ on bounded open solid and compact solid sets in 
the same way as above. Using \cite[Def. 6.1, Th. 3.10]{Butler:TMLCconstr} it is easy to see that  $\mu_{p, \epsilon}$ is a solid-set function, 
so it gives a topological 
measure on $X$. 
\end{example}

\begin{example} \label{basicExDTM}
Let $X$ be a LC space, and let $D$ be a connected compact subset of $X$. Define a set function 
$\nu$ on $\mathscr{O}(X) \cup \mathscr{C}(X)$  by setting $\nu(A) = 1$ if $ D \subseteq A$ and $\nu(A) = 0$ otherwise. 
If $D$ has more than one element, then $\nu$ is a deficient topological measure, but not 
a topological measure. See \cite[Ex. 6.1]{Butler:DTMLC} and \cite[Ex. 1, p.729]{Svistula:DTM} for details.
\end{example}

For  more examples of topological measures and quasi-integrals on LC spaces
see \cite{Butler:TechniqLC}, \cite{Butler:QLFLC}, and the last section of \cite{Butler:TMLCconstr}.
For more examples of deficient topological measures see  \cite{Butler:DTMLC} and \cite{Svistula:DTM}. 

We denote by $ \mathbf{DTM}(X)$ (respectively, $ \mathbf{TM}(X)$) the space of all finite deficient topological measures 
(resp., of all finite topological measures) on $X$.  

\begin{definition} \label{defwk}
The weak topology on $ \mathbf{DTM}(X)$ is the coarsest (weakest) topology for which maps 
$ \mu \longmapsto \mathcal{R}_{\mu} (f), f \in C_0^+(X) $ are continuous.
\end{definition}

\noindent
Unless stated otherwise, all collections of (deficient) topological measures in this paper are assumed to be endowed with the weak topology. 

\begin{remark} \label{rmwkbase}
The basic neighborhoods for the weak topology have the form 
$$N(\nu, f_1, \ldots, f_n, \epsilon) = \{ \mu \in \mathbf{DTM}(X): \ |\mathcal{R}_{\mu}(f_i) - \mathcal{R}_{\nu} (f_i) | < \epsilon, \, f_i \in C_0(X) \},$$
$i=1, \ldots, n  $.
Here we may also take $f_i \in C_0^+(X)$ (see  \cite[Def. 2.1, T. 2.1]{Butler:WkConv} and note that the last two parts of the proof of 
\cite[T. 2.1]{Butler:WkConv}  work for $f \in C_0(X)$).   

Let $\mu_{\alpha}$ be a net in $\mathbf{DTM}(X)$, $\mu \in \mathbf{DTM}(X)$. 
The net $ \mu_{\alpha} $ converges weakly to $ \mu$ (and we write $ \mu_{\alpha} \Longrightarrow \mu$)  
iff  $\mathcal{R}_{\mu_{\alpha}}  (f) \rightarrow  \mathcal{R}_{\mu}  (f)$, i.e. 
$\int f \, d \mu_{\alpha} \rightarrow \int f \, d \mu$  for every $ f \in C_0(X)$ (or every $ f \in C_0^+(X)$).

By \cite[Th. 2.3]{Butler:WkConv}
the weak topology on $\mathbf{DTM}(X)$ is also given by basic neighborhoods of the form 
\begin{align}   
\label{WkNbd}
 W( \nu, U_1, \ldots, U_n, C_1, \ldots, C_m, \epsilon) = \{ \mu: \ \mu(U_i) > \nu(U_i) - \epsilon, \ \mu(C_j) < \nu(C_j) + \epsilon\},  
\end{align} 
where $ i=1, \ldots, n, \ j=1, \ldots m; \nu, \mu  \in \mathbf{DTM}(X);  U_i \in \mathscr{O}(X), C_j \in \mathscr{K}(X); \epsilon >0; n, m \in \mathbb{N}$.
Since the proof of  \cite[Th. 2.3, Th. 2.1]{Butler:WkConv} works just as well for bounded open sets,  in 
(\ref{WkNbd}) we may take $U_i$ to be bounded open sets. 
\end{remark}

\begin{remark} \label{wkwk*}
Our definition of weak convergence corresponds to one used in probability theory. It is the same as a functional analytical definition of  $wk^*$ convergence
on $\mathbf{DTM}(X)$ (respectively, on $\mathbf{TM}(X)$), which is justified by the fact that this topology agrees with the weak$^*$ topology 
induced by p-conic quasi-linear functionals (resp., quasi-linear functionals). 
\end{remark}

\begin{definition} \label{unifbdv}
Let $X$ be LC.
A family $\mathcal{M}  \subseteq \mathbf{DTM}(X)$ is uniformly bounded in variation if there is a positive constant $M$ such that $ \| \mu \| \le M$ 
for each $\mu \in \mathcal{M} $. 
\end{definition}

One important uniformly bounded in variation family is the collection of all (deficient) topological measures satisfying condition $\mu(X) = 1$.

\begin{remark} \label{McompT2} 
By  \cite[Th. 8.7]{Butler:ReprDTM},  $ \mathbf{DTM}(X)$ (respectively, $ \mathbf{TM}(X)$)  is homeomorphic 
to the space of bounded p-conic quasi-linear (resp.,  bounded quasi-linear) functionals endowed with pointwise convergence. 
By Theorem \cite[Th. 2.4]{Butler:WkConv}  $ \mathbf{DTM}(X)$ is Hausdorff.
By Theorem \cite[Th. 2.4, L. 5.3]{Butler:WkConv},  
a uniformly bounded in variation family of (deficient) topological measures  $\{ \mu: \mu(X) \le c \}, c > 0$ is compact Hausdorff.
\end{remark}

\section{Image transformations} \label{SectIT}

\begin{definition}
A function between topological spaces is proper if inverse images of compact subsets are compact.
\end{definition}

If $u :Y \rightarrow X$ is a continuous proper map, then a (deficient) topological measure $ \nu$ on $Y$ 
induces a (deficient) topological measure $ \nu \circ u^{-1}: A \mapsto \nu(u^{-1}(A))$ on $X$ (see Remark \ref{invfunDTM} below, 
\cite[Pr. 5.1]{Butler:DTMLC},  \cite[Ex. 14]{AarnesJohansenRustad}).
In this section we study image transformations and d-image transformations, which generalize this idea. 
They move sets in one space to sets in another space so that (deficient) 
topological measures on the second space induce (deficient) topological measures on the first one.  

\begin{definition} \label{IT}
Let $X$ and $Y$ be LC spaces.
A map $q: \mathscr{O}(X) \cup \mathscr{K}(X) \longrightarrow \mathscr{O}(Y)  \cup  \mathscr{K}(Y)$ such that
\begin{enumerate}[label=(IT\arabic*),ref=(IT\arabic*)]
\item \label{IT1}
$q(U) \in \mathscr{O}(Y)$  for $ U  \in \mathscr{O}(X)$, and $ q(K) \in  \mathscr{K}(Y)$ for $ K \in\mathscr{K}(X)$;
\item \label{IT3}
$q(U) =  \bigcup \{q(K): K \subseteq U, K \in \mathscr{K}(X) \}$ for $ U \in \mathscr{O}(X)$; 
\item \label{IT4}
$ q(K) = \bigcap \{ q(U): K \subseteq U, U \in \mathscr{O}(X) \}$ for $K \in \mathscr{K}(X)$;
\item \label{IT2} 
$q( A \sqcup B)  = q(A) \sqcup q(B)$;   
\end{enumerate}
is called an image transformation (from X to Y) if  $A, B, A\sqcup B \in \mathscr{O}(X) \cup \mathscr{K}(X)$;  $q$ is called a d-image transformation if $A, B \in \mathscr{K}(X)$.
If $X = Y$ we call $q$ a (d-) image transformation on $X$.
\end{definition}

To avoid triviality, we consider (d-) image transformations that are not identically $ \emptyset$. 

\begin{lemma} \label{ITsvva}
Let $X$ and $Y$ be LC spaces.
Suppose $q: \mathscr{O}(X) \cup \mathscr{K}(X) \longrightarrow \mathscr{O}(Y)  \cup  \mathscr{K}(Y)$ is a (d-) image transformation. 
The following holds:  
\begin{enumerate}[leftmargin=0.25in, label=(\roman*),ref=(\roman*)]
\item \label{ITsv0}
$q(\emptyset) = \emptyset$.
\item \label{ITsv1dop}
$q$ is 1-1 if and only if $q(\{x\}) \neq \emptyset $ for each $x \in X$.
\item \label{ITsv1}
$q$ is  monotone, i.e. if $ A \subseteq B$, $A, B \in  \mathscr{O}(X) \cup \mathscr{K}(X) $  then $ q(A) \subseteq q(B)$.
\item \label{ITsv2}
If $ K \subseteq q(U), K \in  \mathscr{K}(Y), U \in \mathscr{O}(X)$ then there exists $ C \in \mathscr{K}(X)$ such that $ C \subseteq U$ and $ K \subseteq q(C) \subseteq q(U)$.
\item \label{ITsv3}
If a net $ U_{t} \nearrow U $, where $U_t, U$ are open sets, then $q(U_{t}) \nearrow q(U)$.
In particuar, $q(U) = \bigcup\{q(V) :  V \subseteq U, V \in \mathscr{O}(X) \}$ for $U \in \mathscr{O}(X)$.
\item \label{ITsv4}
If a net $K_{\alpha} \searrow K$, where $ K_{\alpha},  K$ are compact sets, then $q(K_{\alpha}) \searrow q(K)$.
In particular, $q(K) = \bigcap\{ q(C): K \subseteq C, C \in \mathscr{K}(X)\}$ for $ K \in \mathscr{K}(X)$.
\item \label{ITsv5}
If  $q(K) \subseteq W, K \in \mathscr{K}(X), W \in  \mathscr{O}(Y)$ then there exists $U \in \mathscr{O}(X)$ such that $ K \subseteq U$ and $ q(K) \subseteq q(U) \subseteq q(\overline U) \subseteq W$.
\end{enumerate}
\end{lemma} 

\begin{proof}
(ii). Suppose $q(\{x\}) \neq \emptyset$ for all $x$. Let $ A \neq B, \, A, B \in \mathscr{O}(X) \cup \mathscr{K}(X)$ and $x \in A \setminus B$. 
If $B$ is compact then $ \{x \} \cap B  = \emptyset \Longrightarrow q(\{x\}) \cap q(B) = \emptyset$. The same is true if $B$ is open, 
since  $q(\{x\})$ is disjoint from $q(K)$ for any compact $K \subseteq B$. Thus, $ q(\{x\}) \subseteq q(A) \setminus q(B)$, and $q(A)  \neq q(B)$. 
Conversely,  we may assume that $X$ has more than 1 element, 
and if $q(\{x\}) = \emptyset$ for some $x$ then $q(\{y\}) = q(\{x\} \sqcup \{y\})$, so $q$ is not 1-1.  \\
(iii). Monotonicity is easy to see when $A, B$ are are both open, both compact, or $A  \in \mathscr{K}(X), B \in \mathscr{O}(X)$. 
If $V \subseteq K$, $ V \in \mathscr{O}(X), K \in \mathscr{K}(X)$ then for each open $U \supset K$ we have $q(V) \subseteq q(U)$, and by \ref{IT4} then $q(V) \subseteq q(K)$.\\
(iv). For each  compact $D \subseteq U$ by Lemma \ref{easyLeLC} there is an open set $V$  with compact closure such that $D \subseteq V \subseteq \overline V \subseteq U$. 
Then $q(U) = \bigcup_{D \subseteq U}  q(D)  = \bigcup_{\overline{V} \subseteq U} q(V)$. 
Say, $q(V_1),  \ldots, q(V_n) $ cover $K$. Let $V= V_1 \cup \ldots \cup V_n$.
Then 
$K \subseteq  \bigcup_{i=1}^n  q(V_i) \subseteq q(\overline V)) \subseteq q(U), $
and we may take $C = \overline V$. \\
(v). $\{ q(U_t)\}$ is an increasing net of open sets,  $ q(U_t) \nearrow \bigcup q(U_t)$, and $\bigcup_t  q(U_t) \subseteq q(U)$.
Let $y \in q(U)$. By part \ref{ITsv2} choose compact $C \subseteq U$ with $y \in q(C) \subseteq q(U)$. Say, $U_{t}$ majorizes 
finitely many sets $U_s$ that cover $C$.  Then $y \in q(C) \subseteq  q(U_t)$, so 
$ q(U) \subseteq \bigcup_t  q(U_t)$. \\
(vi). $ \{ q(K_t) \}$ is a decreasing net of compact sets, $ q(K_t) \searrow \bigcap_t q(K_t)$, and  $\bigcap_t q(K_t) \supseteq q(K)$.
Let $U $ be an open set containing $K$. Then there exists $t_0$ such that $ K_t \subseteq U$ for $t \ge t_0$, and  
$ \bigcap_t q(K_t) \subseteq  \bigcap_{t \ge t_0} q(K_t) \subseteq q(U)$. This holds for any $U$ containing $K$, so 
by \ref{IT4}  $ \bigcap_t q(K_t) \subseteq q(K) $. \\
(vii). For $K$ there exists an open set $V$ with compact closure such that $K \subseteq V \subseteq \overline V \subseteq X$.  
Ordered by reverse inclusion,  the 
decreasing net of such sets $ \overline{V_t} \searrow K$, so by part \ref{ITsv4} $q(\overline{V_t} ) \searrow q(K)$. Since $q(K) \subseteq W$,
there is $ t$ such that $q(\overline{V_t})  \subseteq W$. Take $ U = {V_t}$.
\end{proof}

\begin{remark} 
Using Lemma \ref{ITsvva} it is easy to check that the composition of two (d-) image transformations is a (d-) image transformation.  
\end{remark}

\noindent
When $X$ and $Y$ are compact, the definition of an image transformation may be stated as follows:

\begin{definition} \label{ITs}
Let $X$ and $Y$ be compact spaces.
A map $q: \mathscr{O}(X) \cup \mathscr{C}(X) \longrightarrow \mathscr{O}(Y)  \cup  \mathscr{C}(Y)$ is called an image transformation (from X to Y) if 
\begin{enumerate}[label=(\alph*),ref=(\alph*)]
\item \label{IT1s}
$q(U) \in \mathscr{O}(Y)$  for $ U  \in \mathscr{O}(X)$;
\item \label{IT4s}
$q(U) =  \bigcup \{q(K): K \subseteq U, K \in \mathscr{C}(X) \}$ for $ U \in \mathscr{O}(X)$;
\item \label{IT3s}
$q(K) = q(X) \setminus q(X \setminus K)$ for $K \in \mathscr{C}(X)$;
\item \label{IT2s}
$q( U \sqcup V)  = q(U) \sqcup q(V)$ for $U, V \in \mathscr{O}(X)$.
\end{enumerate}
\end{definition}  

\begin{remark}
We introduce the concept of a d-image transformation and extend the concept of an image transformation to LC spaces.
Image transformations between compact spaces 
first appeared in \cite{Aarnes:ITfirst} and were later studied in many papers,
the vast majority of which deals with compact spaces.  Unlike previous works, we don't require $q(X) = Y$ in the definition of an image transformation
(although we show it later in Theorem \ref{ITqh}).
Various previous works have different definitions of an image transformation, but 
with Lemma \ref{ITsvva} (and sometimes easy topological results, such as Lemma \ref{easyLeLC})
we see that our definition of an image transformation 
is equivalent to definitions of an image transformation used, for example,  in  \cite{AarnesGrubb},
\cite{AarnesJohansenRustad},  \cite{OrjanAlf:HomSimpleTM}, \cite{Pedersen}, \cite{AlfImTrans}, \cite{Taraldsen:RegQHinLC}, 
and  \cite{Taraldsen:ITQMonLC}.
In \cite{AarnesTaraldsen}, \cite{Taraldsen:RegQHinLC},  and \cite{Taraldsen:ITQMonLC} 
the definition of an image transformation on LC spaces requires additivity on the collection of open and closed sets, and
hence, mimics the compact spaces case; the results in these papers are similar to ours. 
(Note that  \cite{AarnesTaraldsen},  \cite{Taraldsen:RegQHinLC}, and \cite{Taraldsen:ITQMonLC},
are revised but close versions of works that appeared as NTNU Dept. of Mathematics preprints in 1995-1997.)  
\end{remark}

\begin{example}  \label{invfunIT}
Let $u :Y  \longrightarrow X$ be a continuous proper function.  
If $Y$ is compact, then any continuous $u :Y  \longrightarrow X$ is proper.    
We may also take as proper functions, for example,  (a) $u(y) = \tan  y$, 
where $Y = (-\pi/2, \pi/2), X = \mathbb{R}$ or  (b) $u$ to be a polynomial on $\mathbb{R}$.
Let $q(A) = u^{-1}(A)$ for $A \in \mathscr{O}(X) \cup \mathscr{K}(X)$. 
Then $q$ is a (d-) image transformation.
Below in Theorem \ref{ITfromUf} we characterize (d-) image transformations that are inverses of continuous proper functions.  
\end{example}

\begin{definition}
A (p-conic) quasi-linear functional is simple if it is a quasi-integral with respect to a simple (deficient) topological measure.
Let $X^{\flat}$ and $X^\sharp$ be, respectively, the collections of simple deficient topological measures on $X$, and simple topological measures on $X$.  
For $ A \in \mathscr{O}(X) \cup \mathscr{C}(X)$ let $A^{\flat} = \{ \mu \in X^{\flat}: \mu(A) = 1 \}$, and $A^\sharp = \{ \mu \in X^\sharp: \mu(A) = 1 \}$.
\end{definition}

\begin{theorem}  \label{rhosimple}
Let $X$ be LC.  \\
(A)  For a quasi-linear functional $\rho$ on $ C_0(X)$ the following are equivalent:
\begin{enumerate}[label=(\roman*),ref=(\roman*)]
\item \label{sim} 
$\rho$ is simple. 
\item \label{ptm}
$m_f$ is a point mass at $y = \rho(f) \in f(X)$.
\item \label{phiout}
$ \rho(\phi \circ f) = \phi(\rho(f)) $ for any $\phi \in C([a,b])$ (with $\phi(0) = 0$ if $X$ is LC but not compact), where $f(X) \subseteq [a,b]$.
\item \label{mult}
$\rho$ is multiplicative on each singly generated subalgebra, 
i.e. for each  $f \in C_0(X)  $ we have:
$\rho(gh) =  \rho (g) \, \rho (h)$ for $g,h \in B(f)$. 
\item \label{fsquare}
$\rho(f^2) = (\rho(f))^2$  for each  $f \in C_0(X) $. 
\end{enumerate}
(B)  For a p-conic quasi-linear functional $\mathcal{R}$ on $ C_0(X)$ the following are equivalent:
\begin{enumerate}[label=(p\roman*),ref=(p\roman*)]
\item \label{sim1} 
$\mathcal{R}$ is simple. 
\item \label{ptm1}
$r_f$ is a point mass at $y = \mathcal{R}(f) \in f(X)$.
\item \label{phiout1}
$ \mathcal{R}(\phi \circ f) = \phi(\mathcal{R}(f)) $ for any nondecreasing $\phi \in C([a,b])$ 
(with $\phi(0) = 0$ if $X$ is noncompact), where $f(X) \subseteq [a,b]$.
\item \label{mult1}
$\mathcal{R}$ is multiplicative on each singly generated cone, 
i.e. for each  $f \in C_0(X)  $ we have:
$\mathcal{R}(gh) =  \mathcal{R} (g) \, \mathcal{R} (h)$ for $g,h \in A^+(f)$. 
\item \label{fsquare1}
$\mathcal{R}(f^2) = (\mathcal{R}(f))^2$  for each  $f \in C_0(X) $.
\end{enumerate}
\end{theorem}

Part (A) (replacing $C_c(X)$ with $C_0(X)$) is proved in \cite[Th. 3.1]{Butler:RepeatedQint}.
Moreover, replacing  quasi-linear functionals with p-conic quasi-linear functionals,
we see that the proof of part (B)  
repeats almost verbatim the proof of \cite[Th. 3.1]{Butler:RepeatedQint} (use \cite[Th. 6.7]{Butler:ReprDTM} instead of formula 
$m_f(W) = \mu(f^{-1}(W))$  
in the proof of (i) $ \Longrightarrow $ (ii), and \cite[L. 7.8]{Butler:ReprDTM} in 
(ii) $ \Longrightarrow $ (iii) in \cite[Th. 3.1]{Butler:RepeatedQint}). 

\begin{lemma} \label{XdiezCom}
Let $X$ be LC. The spaces  $X^{\flat}$ and $X^\sharp$ are compact Hausdorff, and so are $C^{\flat}$ and $C^\sharp$ for any compact $C$.
If $U$ is a nonempty open set in $X$, then  $U^{\flat}$ is a nonempty open set in $X^{\flat}$, and  $U^\sharp$ is a nonempty open  set in $X^\sharp$. 
\end{lemma} 

\begin{proof}
Let $ \mu_t \Longrightarrow \mu$, where $ \mu_t \in  X^{\flat}$. For any $f \in C_0(X)$ by Theorem \ref{rhosimple} we have:
$ \int_X f^2 \, d\mu = \lim  \int_X f^2 \, d\mu_t =   \lim ( \int_X f \, d\mu_t )^2 = (\int_X f \, d\mu)^2$, which shows that $ \mu$ is simple. Thus, 
$(X^{\flat})$ is closed. 
By Remark \ref{McompT2},  $X^{\flat}$ and $X^\sharp$  are compact Hausdorff. 

Let $C$ be compact.  Suppose $ \mu_t \in C^\flat,  \mu_t \Longrightarrow \mu$. We claim that $ \mu \in C^\flat$. 
Suppose to the contrary that $\mu(C) = 0$.  Then  $ \mu(U) = 0$ for some open $U \supset C$. 
By Lemma \ref{easyLeLC} there is an open set $V  \supset C$ with $ \mu(V) = \mu(\overline V) = 0$. 
But $\mu_t(V) = 1$ for all $t$, and by \cite[Th. 2.1(3)]{Butler:WkConv}  $\mu(V) =1$. The contradiction shows that  $ \mu \in C^\flat$. 
So $C^{\flat}$ (and similarly $C^\sharp$) is closed, hence, compact. 

If $U$ is a nonempty open set in $X$, then $U^{\flat} = \{ \mu \in X^{\flat}: \mu(U) = 1 \} = \{ \mu \in X^{\flat}: \mu(U) > \delta_p (U) - \epsilon\} $, 
where $p \in U, 0 < \epsilon< 1$.  By Remark \ref{rmwkbase} $U^{\flat}$ is a nonempty open set in $X^{\flat}$. 
Similarly, $U^\sharp$ is a nonempty open  set in $X^\sharp$.
\end{proof}

\begin{remark} \label{basicNbd*}
If $\nu$ is a simple deficient topological measure, and $ \nu (U_i) = 1, \nu (C_j) = 0$, where $U_1, \ldots, U_n$ are (bounded) open and 
$C_1, \ldots, C_m$ are compact, then by Remark \ref{rmwkbase} 
the basic neighborhood $W$ of $\nu$ in $X^{\flat}$
for any $\epsilon < 1$ has the form:
 $$W =W(U_1, \ldots, U_n, C_1, \ldots, C_m) = \{ \mu \in X^{\flat}: \mu(U_i)= 1, \, \mu(C_j) = 0 \},$$  $i=1, \ldots, n,  j =1, \ldots, m$.
If $\nu$ is a simple topological measure,  then $\mu(C) = 0$ (where $C$ is compact) iff $\mu(X \setminus C) = 1$, 
so a basic neighborhood in $X^\sharp$ has the form:
\begin{align} \label{XstTopo}
  W = W\{ U_1, \ldots, U_n \} = \{ \mu \in X^\sharp: \mu(U_i)= 1, \,  i=1, \ldots, n \} = U_1^\sharp \cap \ldots \cap U_n^\sharp.
\end{align} 
\end{remark}

Let $\tau_s$ be the  topology on $\mathbf{DTM}(X)$ with basic open sets of the form (\ref{WkNbd}), where $U_i$ are bounded open solid sets, and 
$C_j$ are compact solid sets.  

\begin{lemma} \label{homeom} 
Suppose $X$ is a LC connected, locally connected space, and family  $\mathcal{M}  \subseteq \mathbf{TM}(X)$ is 
closed and uniformly bounded in variation. 
Then the identity map $i: \mathcal{M} \longrightarrow (\mathcal{M}, \tau_s)$ is a homeomorphism.
In particular, basic neighborhoods of $X^\sharp$  can be given by (\ref{XstTopo}), where sets $U_i$  are bounded open solid.
\end{lemma}

\begin{proof}
Clearly, 
$i: \mathcal{M} \longrightarrow (\mathcal{M}, \tau_s)$ is continuous. By Remark \ref{McompT2} $  \mathcal{M}$ is compact.  
By \cite[Th. 10.7]{Butler:TMLCconstr} a topological measure is uniquely determined by its values on 
compact solid sets and open bounded solid sets. Thus, if $\mu \neq \nu, \, \mu, \nu \in \mathcal{M}$ 
then there is a compact solid or a bounded open solid  set $A$ 
such that $ \nu(A)  \neq \mu(A)$, and so $(\mathcal{M}, \tau_s)$ is Hausdorff. The statement now follows.
\end{proof}

\begin{example} \label{staIT}
Let $X$ be LC. For $A \in \mathscr{O}(X) \cup \mathscr{K}(X)$ let $\Lambda(A) = A^{\flat}$ (respectively,  $\Lambda(A) = A^\sharp$).
From Lemma \ref{XdiezCom} and elementary properties of simple (deficient) topological measures it is easy to see that
 $\Lambda$ is a d-image transformation from $X$ to $X^{\flat}$ (resp., an image transformation from $X$ to $X^\sharp$).
For instance, if $A, B, A \sqcup B \in \mathscr{O}(X) \cup \mathscr{K}(X)$ and $ \mu$ is simple,
$\mu(A \sqcup B) = 1$  implies that $ \mu(A) =1, \mu(B) =0$ or  $ \mu(A) =0, \mu(B) =1$, so $(A \sqcup B)^\sharp = A^\sharp \sqcup B^\sharp$.
\end{example}

\noindent
We call $\Lambda $ the basic d-image transformation (respectively, the basic image transformation).

\begin{theorem} 
Let $q$ be a (d-) image transformation  from $X$ to $Y$;  $X$, $Y$ are LC. 
For  a (deficient) topological measure $ \nu$ on $Y$ define the set function on $X$ by $(q^*\nu)(A) = \nu(q(A))$ 
for $ A \in \mathscr{O}(X) \cup \mathscr{K}(X)$
and $(q^*\nu)(F) = \inf\{ (q^*\nu) (U) : F \subseteq U, U \in \mathscr{O}(X)\}  $ for $ F \in \mathscr{C}(X)$. 
Then $q^*\nu$ is  a (deficient) topological measure on $X$.
\end{theorem} 

\noindent
We call $q^* : TM(Y) \rightarrow TM(X)$ (resp.,   $q^* : DTM(Y) \rightarrow DTM(X)$) the adjoint map for $q$.
 
\begin{proof}
The first property of Definition \ref{TMLC} follows from Definition \ref{IT}\ref{IT2}. For $U  \in \mathscr{O}(X)$, $ q(U)$ is open in $Y$, and we pick a compact 
$ K \subseteq q(U)$ such that $ \nu(q(U)) - \nu(K) < \epsilon$, if (a) $ \nu(q(U)) < \infty$ and $ \epsilon >0$; or such that $\nu(K) > c$, if  (b) $ \nu(q(U))  = \infty, c>0$.
By Lemma \ref{ITsvva}\ref{ITsv2} find $C \in \mathscr{K}(X)$ with $K \subseteq q(C) \subseteq q(U)$. 
In case (a) $ \nu(q(U)) - \nu(q(C)) \le \nu(q(U)) - \nu(K) < \epsilon$, and in case (b) $ \nu(q(C)) \ge \nu(K) >c$.  The second property 
of  Definition \ref{TMLC} follows. 

We shall show that  $q^*\nu $ is consistently defined on $ \mathscr{K}(X)$. Let $K \in \mathscr{K}(X)$.
If  $ \nu(q(K)) < \infty$,  pick an open set $ W \supseteq q(K)$ such that $\nu(W) - \nu(q(K)) < \epsilon$ for $ \epsilon >0$.  
By Lemma \ref{ITsvva}\ref{ITsv5} pick an open set $U \supseteq K$ 
such that $q(K) \subseteq q(U) \subseteq W$. Then $ \nu(q(U)) - \nu(q(K)) \le \nu(W) - \nu(q(K)) < \epsilon$, so  
$(q^*\nu) (K) = \inf \{(q^*\nu)(U): K \subseteq U, U \in \mathscr{O}(X)\}  $. The last equality is also easily seen if $ \nu(q(K)) = \infty$.
\end{proof}

\begin{remark}
It is easy to verify that $(p \circ q)^* = q^* \circ p^*$ for 
(d-) image transformations $p$ and $q$.
 We see that (d-) image transformations on $X$ are morphisms in a category with objects $\mathscr{O}(X) \cup \mathscr{K}(X)$, and adjoints of  
(d-) image transformations on $X$ are morphisms in a category with objects (deficient) topological measures on $X$.
\end{remark}

\begin{remark} \label{invfunDTM}
If $X,Y$ are LC, $u :Y \rightarrow X$ is a continuous proper map, and $ q = u^{-1}$ as in Example \ref{invfunIT}, 
we see that a (deficient) topological measure $ \nu$ on $Y$ 
induces a (deficient) topological measure $\mu =q^*\nu =  \nu \circ u^{-1}: A \mapsto \nu(u^{-1}(A))$ on $X$.
\end{remark}

\begin{theorem} \label{adjCont}
Suppose $X,Y$ are LC and $q$ is a d-image transformation (respectively, an image transformation) from $X$ to $Y$.
Then the adjoint map $q^* :DTM(Y) \rightarrow DTM(X)$ (resp.,  $q^* : TM(Y) \rightarrow TM(X)$) preserves finite linear combinations and  
is continuous.
In particular,  $q^* :Y^{\flat} \rightarrow X^{\flat}$ (resp.,  $q^* :Y^\sharp \rightarrow X^\sharp$)  is continuous.
If $y \in q(\{ x\})$ then $q^*(\delta_y) = \delta_x$.
\end{theorem}

\begin{proof}
Suppose  $\nu_{\alpha} \Longrightarrow \nu_0$ in  $ DTM(Y)$. We shall show that  $\mu_{\alpha} \Longrightarrow \mu_0$ in  $ DTM(X)$, 
where $ \mu_{\alpha} = q^*(\nu_{\alpha}), \mu_0 = q^*(\nu_0)$.
Let $W =W(\mu_0, U_i, C_j, \epsilon)$, $i=1, \ldots n, j= 1, \ldots m$ be a basic neighborhood of $ \mu_0$ of the form (\ref{WkNbd}). 
For a  basic neighborhood $ N= N(\nu_0, q(U_i), q(C_j), \epsilon), i=1, \ldots n, j= 1, \ldots m$ there is $\alpha'$ such that 
$\nu_{\alpha} \in N$ for any $ \alpha >\alpha'$. Then $\mu_{\alpha} = q^*(\nu_{\alpha}) \in W$ for such $\alpha$, 
so  $\mu_{\alpha} \Longrightarrow \mu_0$, and $q^*$ is continuous.   
Clearly, the adjoint map $q^*$ maps $Y^{\flat}$ into $X^{\flat}$, so 
 $q^* :Y^{\flat} \rightarrow X^{\flat}$ (respectively,  $q^* :Y^\sharp \rightarrow X^\sharp$)  is continuous.
 
Suppose  $y \in q(\{ x\})$. So $\mu =q^*(\delta_y)$ is a simple (deficient) topological measure. Since
$\mu(\{x\}) = q^*(\delta_y) (\{x\}) = \delta_y(q(\{x\})) = 1$, by Remark \ref{dtmisptm} $\mu = \delta_x$.   
\end{proof}

\noindent
We denote by $i_Y: Y \rightarrow  \mathbf{DTM}(Y) $ a continuous map $y \longmapsto \delta_y$.

\begin{definition}
Let $X, Y$ be LC spaces. 
We say that $w: Y \rightarrow X^{\flat}$ (respectively, $w: Y \rightarrow X^\sharp$) is k-proper if $w^{-1}( K^\flat)$ (resp.,  $w^{-1}( K^\sharp)$) 
is compact for any $K \in \mathscr{K}(X)$.
\end{definition}

\begin{remark} \label{easyKprpr}
Suppose $Y$ is compact.  Since $X^{\flat}$ is Hausdorff, any continuous $w: Y \rightarrow X^{\flat}$  (or  $w: Y \rightarrow X^\sharp$) is proper, 
hence, k-proper. Suppose $Y$ is LC, connected, locally connected. Then for a compact $K$ there is a compact solid set $C$ containing $K$ 
(see \cite[L. 2.5, L. 3.7]{Butler:TMLCconstr}), so $w^{-1}(K) \subseteq w^{-1}(C)$. Thus, to show that $w$ is k-proper, it is enough to check that 
$w^{-1}( C^\flat)$ (respectively,  $w^{-1}( C^\sharp)$)  is compact for any compact solid set $C$. 
\end{remark}

\begin{theorem} \label{ITw}
Let $X$  and $Y$ be LC spaces. \\
(I). If $q$ is a d-image transformation (respectively, an image transformation) from $X$ to $Y$ 
then there exists a continuous k-proper function $w: Y \longrightarrow X^{\flat}$ (resp.,   $w: Y \longrightarrow X^\sharp$) 
such that $q = w^{-1} \circ \Lambda$; namely, $w = q^* \circ i_Y$.
Here $\Lambda$ is the basic d-image transformation (resp., the basic image transformation).   \\
(II). If $w: Y \longrightarrow X^{\flat}$ (resp., $w: Y \longrightarrow X^\sharp$) is continuous k-proper,  
then $q = w^{-1} \circ \Lambda$ is a d-image transformation
(resp., an image transformation).  We have $w = q^* \circ i_Y$.

Thus, there is a 1-1 correspondence between d-image transformations  (respectively, image transformations)  from $X$ to $Y$  and 
continuous k-proper functions $w: Y \longrightarrow X^{\flat}$ (resp.,   $w: Y \longrightarrow X^\sharp$).
\end{theorem}

We may write $w_y$ for the simple (deficient) topological measure $w(y)$.

\begin{proof}
The proof is basically identical for d-image transformations and image transformations.
We shall show the first part for d-image transformaitons, and the second part for image transformations. \\
(I). By Theorem \ref{adjCont}  $q^*$ is continuous, and so is $w = q^* \circ i_Y: Y \rightarrow X^{\flat}$. 
For $A \in \mathscr{O}(X) \cup \mathscr{K}(X)$,  
$y \in q(A) \Longleftrightarrow \delta_y (q(A)) =1 \Longleftrightarrow q^* \delta_y \in A^{\flat}  \Longleftrightarrow w(y) \in A^{\flat}  
\Longleftrightarrow y \in w^{-1} (A^{\flat})$. Hence, $q = w^{-1} \circ \Lambda$. 
Also, $ w^{-1} (K^{\flat}) = q(K)$ is compact for $K \in \mathscr{K}(X)$. Thus, $w$ is k-proper.   \\
(II). Clearly, $q(K)$ is compact.
To verify Definition \ref{IT}\ref{IT3}, 
take $y \in q(U) = w^{-1}(U^\sharp)$. Then $w_y(U) = 1$, so there is a compact $K \subseteq U$ with $w_y (K) =1$, i.e. $ y \in  w^{-1}(K^\sharp)$. 
Then $q(U)  =  w^{-1}(U^\sharp) \subseteq \bigcup_{K \subseteq U} w^{-1}(K^\sharp) = \bigcup_{K \subseteq U} q(K) \subseteq q(U)$, which gives  part \ref{IT3}. 
If $K \in \mathscr{K}(X)$ and $y \in q(K),$ then $w_y(K) =1$, and so for any open $U$ containing $K$,  $w_y(U) = 1$, i. e. $ y \in w^{-1}(U^\sharp) = q(U)$.  
\ref{IT4} follows. For Definition \ref{IT}\ref{IT2} see Example \ref{staIT}. 
For $A \in \mathscr{O}(X) \cup \mathscr{K}(X)$  and $y \in Y$ we have 
$(q^* \delta_y)(A) = \delta_y(q(A)) = \delta_y(w^{-1}(A^\sharp)) =1$ iff $w(y) \in A^\sharp \Longleftrightarrow w(y)(A) =1$. 
This gives $ q^* \circ i_Y = w$.
\end{proof}

\begin{example} \label{consQst}
Let $X$ be LC, $Y$ be compact, and  $\mu$ be a simple (deficient) topological measure on $X$. 
For $A \in \mathscr{O}(X) \cup \mathscr{K}(X)$ let $q(A) = \emptyset$ if $ \mu(A) = 0$, and let $q(A) = Y$ if $ \mu(A) = 1$.
It is easy to see that $q$ is a (d-) image transformation from $X$ to $Y$. If $ \nu$ is a normalized (deficient) topological measure on $Y$ 
then for any open set $V$ in $X$,
$(q^*\nu)(V) = \nu(q(V)) = \nu(Y) = 1$ if $ \mu(V) =1$, and $(q^*\nu)(V) = 0$ if $ \mu(V) =0$. In other words, $q^*\nu = \mu$.
Thus, on  normalized (deficient) topological measures the adjoint map $q^*$ is a constant map.
\end{example} 

\begin{theorem} \label{ITfromUf}
Suppose $q$ is a (d-) image transformation from $X$ to $Y$, where $X$, $Y$ are LC. TFAE:
\begin{enumerate}[label=(\roman*),ref=(\roman*)]
\item
$q = u^{-1}$ for some proper continuous function $u :Y  \longrightarrow X$.
\item \label{yqx}
$Y = \bigcup_{x \in X} q(\{x \})$.
\item
$q^*(P_e(Y)) \subseteq P_e(X)$.
\end{enumerate}
\end{theorem}

\begin{proof}
(i) $\Longrightarrow$ (ii). For $y \in Y$ let $x =u(y)$. Then $y \in u^{-1}(x) =q(\{x \})$.
(ii) $\Longrightarrow$ (iii).  
Let $y \in Y$. Suppose $y \in q(\{x\})$. By Theorem \ref{adjCont} $ q^*(\delta_y) = \delta_x$.
 (iii) $\Longrightarrow$ (i).  
 For $y \in Y$,  $ q^*(\delta_y) = \delta_x$, and we let $u(y) = x$. That is, $u :Y  \longrightarrow X$ is $u =  i_X^{-1} \circ w$, where 
 $w =q^* \circ i_Y$ is a continuous k-proper function from $Y$ to $X^\sharp$ (or to $X^{\flat}$) from Theorem \ref{ITw}. 
If $y_\alpha \rightarrow y$, then $\delta_{x_\alpha} = q^*(\delta_{y_\alpha })$ converges to $ q^*\delta_y = \delta_x$, 
so $x_\alpha \rightarrow x$,  showing that $u$ is continuous.  
 For a compact $K$ in $X$, $u^{-1}(K) = w^{-1} \circ i_X (K) \subseteq  w^{-1}(K^\sharp)$ (respectively,
 $u^{-1}(K) \subseteq  w^{-1}(K^{\flat})$), a compact set, so $u$ is proper.   By (iii) $w(Y) \subseteq P_e(X)$, so 
 $q(A) = w^{-1}(A^\sharp) = w^{-1}(A^\sharp \cap P_e(X))  (\mbox{resp.}, w^{-1}(A^\flat \cap P_e(X)) )
 =  w^{-1}(i_X(A))   = u^{-1}(A)$ for $A \in \mathscr{K}(X) \cup \mathscr{O}(X)$. Thus, $q= u^{-1}$.
 \end{proof}
  
\noindent
The next theorem provides a tool for constructing image transformations. 
 
\begin{theorem}  \label{ITsolid}
Let $X$ be a LC connected locally connected space, $Y$ be a LC space. 
Suppose $q: \mathscr{A}_{s}^{*}(X) \longrightarrow  \mathscr{O}(Y) \cup \mathscr{K}(Y)$ such that:
\begin{enumerate}[label=(\roman*),ref=(\roman*)]
\item \label{ITsol1}
$ \bigsqcup_{i=1}^n q(C_i) \subseteq q(C)$ whenever $\bigsqcup_{i=1}^n C_i \subseteq C,  \ \  C, C_i \in \mathscr{K}_{s}(X)$; 
\item  \label{ITsol2}
$q(U) \in \mathscr{O}(Y)$ and  $q(U) = \bigcup \{ q(K): \ K \subseteq U , \ K \in \mathscr{K}_{s}(X) \}$ for $U \in \mathscr{O}_{s}^{*}(X)$; 
\item \label{ITsol3}
$q(K) \in \mathscr{K}(Y)$ and  $q(K) = \bigcap \{ q(U) : \  K \subseteq U, \ U \in \mathscr{O}_{s}^{*}(X) \}$ for $ K  \in \mathscr{K}_{s}(X)$; 
\item  \label{Qsolidparti}
$ q(A) = \bigsqcup_{i=1}^n q (A_i)$ whenever $A = \bigsqcup_{i=1}^n A_i, \ \ A , A_i  \in \mathscr{A}_{s}^{*}(X)$.
\end{enumerate}
Then $q$ extends uniquely to an image transformation from $X$ to $Y$. 
\end{theorem}

\begin{proof}
Let $y \in Y$. Define a set function $w_y$ on  $\mathscr{A}_{s}^{*}(X)$ by $w_y(A) = \delta_y(q(A))$. 
By Definition \ref{DeSSFLC} it is easy to check that $w_y$ is a solid-set function on $X$, hence it extends to a unique simple topological 
measure on $X$, which we also call $w_y$. Thus, we obtain a map $w:Y \longrightarrow X^\sharp$, where $w(y) = w_y$.
If $ A \in \mathscr{A}_{s}^{*}(X)$ then 
\begin{align} \label{qaso}
w^{-1}(A^\sharp) = \{ y: w_y(A) = 1\} = \{ y: y \in q(A) \} = q(A).
\end{align}
By Lemma  \ref{homeom} we may consider subbasic open sets in $X$ of the form $U^\sharp$ for $U \in \mathscr{O}_s^*(X)$. Then
$w^{-1}(U^\sharp) = q(U)$ is an open set, and  $w$ is continuous.   
By Remark \ref{easyKprpr} $w$ is k-proper. 
By Theorem \ref{ITw}, $w$ gives a unique image transformation from $X$ to $Y$, which by (\ref{qaso})
coincides with $q$ on  $\mathscr{A}_{s}^{*}(X)$.
\end{proof} 

\begin{remark} \label{easysolpar}
If $X$ is compact,  in Theorem \ref{ITsolid} it is enough to consider one of the equivalent conditions \ref{ITsol2} and \ref{ITsol3}.  
If $X$ is a noncompact LC connected locally connected space whose one-point compactification has genus 0, then
Theorem \ref{ITsolid}\ref{Qsolidparti} holds automatically by \cite[L. 15.2]{Butler:TMLCconstr}.  
If $X$ is a q-space, then for Theorem \ref{ITsolid}\ref{Qsolidparti} we only need to show that 
$q(X) = q(A) \sqcup q(X \setminus A)$ for partitions  $X = A \sqcup (X \setminus A)$, where $A$ is a closed solid or an open solid set.
(First, replacing $\lambda$ by $q$, 
a minor modification of the argument of \cite[L. 12.3, part 4.]{Butler:TMLCconstr} shows that 
it is enough to consider so-called irreducible partitions of $X$, but by \cite[Rem. 11.6 (a4)]{Butler:TMLCconstr} any such partition has the form 
$X = A \sqcup (X \setminus A)$. See \cite[Sect.11]{Butler:TMLCconstr} for more detail.) 
\end{remark}

\begin{example} \label{3pts2vIT}
Let $X = Y$ be the unit square in $\mathbb{R}^2$. Let $E = \{x, z\}$, where $x, z\in X$. As in \cite[p. 46]{AarnesGrubb}, we consider the following map $q$.
For an open or closed solid set $A$ define $q(A) = 0$ if $| A \cap E| = 0$; $q(A) =A $ if $| A \cap E| = 1$; $q(A) = X$ if $| A \cap E| = 2$. 
It is easy to see that $q$ satisfies Theorem \ref{ITsolid}, hence, extends to an image transformation from $X$ to $X$.
Calculating $q(A)$ for a nonsolid set $A$ mimics finding a topological measure on nonsolid sets and uses topological facts, 
see \cite{Aarnes:ConstructionPaper},  \cite[Rem. 6.3]{Butler:TMLCconstr}.  
For instance, if $C$ a closed connected set, then connected components $V_{\alpha}$ of  $X \setminus C$ are open solid sets (\cite[L. 3.2]{Aarnes:ConstructionPaper}), 
and from Lemma \ref{ITsvva} we see that $q(X \setminus C) = \bigcup q(V_{\alpha}).$  
If $x, z \in C$ then $q(X \setminus C) = \emptyset$, and $q(C) = X$.
If $x, z \notin C$, and 
$x, z$ are in the same component of $X \setminus C$, then $q(X \setminus C) =X$, and $q(C) = \emptyset$.
If $x, z \notin C$, and 
$x, z$ are not in the same component of $X \setminus C$, then $q(X \setminus C)$ is the union of two components that contain $x$ and $z$, hence, 
$ q(C)$ is the union of $C$ with those components of $X \setminus C$ which contain neither $x$ nor $z$. 

Let $w: X \longrightarrow X^\sharp$ be the corresponding continuous k-proper function. For $y \in Y$, the simple topological measure $w_y$ 
acts on solid (open or closed) sets as follows: 
$w_y(A) = 1$ iff $y \in w^{-1}(A^\sharp) $ iff $y \in q(A)$. This means that $w_y(A) = 1$ iff $| A \cap E| = 2$; or $| A \cap E| = 1$ and $y \in A$. 
Thus, $w_y$ is the topological measure from Example \ref{nvssf} for the set $P = \{ x,z,y\}$. 
If $ \mu $ is a finite topological measure on $X$, the adjoint $q^*$  gives a new topological measure $q^*\mu$, which acts on $A \in \mathscr{A}_{s}^{*}(X)$ 
as follows: $(q^*\mu)(A) = 0$ if $| A \cap E| = 0$;  $(q^*\mu)(A) = \mu(A)$ if $| A \cap E| = 1$;  $(q^*\mu)(A) = \mu(X)$ if $| A \cap E| = 2$.
\end{example}

\begin{example} \label{ITAarnesTM}
Let $X = Y$ be the unit disk in $\mathbb{R}^2$, and let $B$ be a closed subset of $X$.  
For  open solid or closed solid set $A$ define 
$q(A) = 0$ if $ A \cap B = \emptyset$;  $q(A) = X$ if $B \subseteq A$; and $q(A) =A $ otherwise. 
(With $B$ consisting of two points we obtain Example \ref{3pts2vIT}.) 
Then $q$ extends to an image transformation on $X$.
(This and similar examples appear in several papers, including 
\cite{OrjanAlf:HomSimpleTM}, \cite{AarnesGrubb}, \cite{Aarnes:ITfirst}.)  
If $ \mu \in B^\sharp$ then $q^*(\mu) = \mu$, and $q^*$ is a retraction of $X^\sharp$ onto $B^\sharp$ (see \cite[L. 28. Th. 29]{OrjanAlf:HomSimpleTM}). 
Using such image transformations one can show that if $ A \in \mathscr{O}(X) \cup \mathscr{C}(X)$ is connected then $A^\sharp$ is also connected for any q-space $X$ 
(see \cite[Cor. 32]{OrjanAlf:HomSimpleTM}).

Let $B$ be the boundary of $X$. 
For a closed connected  set $C$ which intersects, but does not contain $B$,  reasoning as in Example \ref{3pts2vIT} 
shows that $q(C)$ is the union of $C$ with connected components of $ X \setminus A$ which do not intersect $B$.   
Let $w: X \longrightarrow X^\sharp$ be the continuous function corresponding to $q$. 
Let $y \in Y$. Since $w_y(A) = 1$  iff $y \in q(A)$, we see that 
$w_y$ is the Aarnes circle topological measure from Example \ref{Aatm} with $p = y$. 
Any closed set $K$ such that does not intersect the boundary $B$ and has diameter $ diam(K) < 1/2$ is contained in a disk $D$  with $diam (D) < 1$. 
Then $w_y(D) = 0$ for $y \in D$, and $q(D) = \emptyset$. 
This shows that $q$ (and any topological measure of the form $q^*\mu$) 
annihilates any closed set that does not intersect the boundary $B$ and has diameter less than $1/2$. 
\end{example} 

\begin{example}  \label{3pts2vW}
Let $X = Y = \mathbb{R}^2$. For $(x, y) \in \mathbb{R}^2$ 
let $\mu_{x, y} $ be the simple topological measure on $X$ constructed as in Example \ref{nvssf} using points
$(x,0), (x,y), (y,0)$. (A similar example when $X$ is compact is \cite[Ex. 3.7(c)]{Aarnes:ITfirst}.)
Using  Lemma  \ref{homeom} and Remark \ref{easyKprpr} it is not hard to check that  
$w: (x,y) \rightarrow \mu_{x, y}$ of $Y$ into $X^\sharp$ is continuous and k-proper.  
By Theorem \ref{ITw} $q = w^{-1} \circ \Lambda$ is an image transformation. Note that $q(A) = \emptyset$ for any bounded open solid set or
compact solid set $A$ which does not intersect the axes. 
\end{example}

\begin{example} \label{ITqfunction}
Let $X$ be a square in $\mathbb{R}^2$. Let $m$ be the normalized Lebesgue measure on $X$. For $ \epsilon < 1/2$ 
define the map $q_{\epsilon}$ on closed solid sets by
$q_{\epsilon}(C) = 0$ if $m(C) < \epsilon$;  $q_{\epsilon}(C) = C$ if $\epsilon \le m(C) < 1-\epsilon$; $q_{\epsilon}(C) = X$ if $m(C) \ge 1- \epsilon$. 
As shown in \cite[L. 51]{AarnesJohansenRustad}, 
$q_{\epsilon}$ extends to an image transformation on $X$. Note that $q_{\epsilon}^* m$ 
is the topological measure $ h \circ m$, constructed using a q-function $h$, 
where $h(t) = 0$ for $ t \in (0, \epsilon)$;  $h(t) = t$ for $ t \in [\epsilon, 1 - \epsilon)$; and 
$h(t) = 1$ for $ t \in [1-\epsilon, 1]$. (See \cite[Th. 3.5, Def. 3.1]{QfunctionsEtm} for details). 
This example works for any q-space $X$ and any normalized topological measure on $X$ 
whose split spectrum does not include $\epsilon $, see \cite[L. 51, Def. 6, Def. 11]{AarnesJohansenRustad} and \cite[Sect. 3]{QfunctionsEtm}.
\end{example} 

\begin{example} \label{HaarNot1} 
Let $ X = Y = \mathbb{R}^2$, and let $ \mu_{p, \epsilon}$ be a topological measure on $X$ from Example \ref{Aatm}. 
Let $w_\epsilon:Y \longrightarrow X^\sharp$  be given by $w_\epsilon(p) = \mu_{p, \epsilon}$. 
By Lemma  \ref{homeom} and Remark \ref{easyKprpr} we see that
$w_\epsilon$ is a continuous $k$-proper map, hence, $q_\epsilon = w_\epsilon^{-1} \circ \Lambda$ is an image transformation on $\mathbb{R}^2$, and 
$q_\epsilon(A) = \{p: \mu_{p, \epsilon}(A) = 1\}$. For $\epsilon = 0$ we have $q_0(A) = A$.  
Let $ \epsilon >0$. 
For  a closed square $A$,  $q_\epsilon(A) $ is four corners of $A$ if $ diam(A) = \epsilon$, and $q_\epsilon(A) = A$ if  $ diam(A) \ge 2 \epsilon$. 
If $A$ is any set with  $ diam(A) < \epsilon$ then $q_\epsilon(A) =\emptyset$.  
Image transformations that "kill" sets of certain size intuitively reflect the fact that when we deal with images infinite resolution 
is not realistic.

Let $s_t:X \longrightarrow X$ be the translation $s_t(x) = x+ t$. We see that
$q_{\epsilon} \circ s_t^{-1} =  s_t^{-1} \circ  q_{\epsilon} $. 
Using the Lebesque measure $m$ on $X$, we obtain a family $ \{ q_\epsilon^* m = m \circ q_\epsilon, \  \epsilon >0 \}$
of invariant topological measures on $X$. Thus, uniqueness of Haar measure on LC topological group 
does not hold in the more general context of topological measures, compared to measures.
This example generalizes \cite[Ex. 3.7(d)]{Aarnes:ITfirst}  and \cite[Ex. D]{AarnesTaraldsen}. 
The non-uniqueness of Haar measures is also in \cite[p. 2072]{Grubb:IrrPart}. 
\end{example}

\begin{example}
Let $G$ be a compact group, and $\mu$ be a simple deficient topological measure on $G$. With $q(A) = \{ g \in G : \mu(g^{-1}A) = 1\}$, as in 
\cite[Ex.1]{AarnesGrubb} we see that $q$ is a d-image transformation from $G$ to $G$, and if $\lambda$ is a left Haar measure on $G$ 
then $ \nu = q^*\lambda$ is not a Haar measure on $G$.
\end{example}

\begin{remark}
There is a generalization of the distribution of the sample median which is 
equivariant with respect to a large collection of transformations. This generalization
is given by the adjoint of an image transformation.  See   \cite{AlfMultidimMedian}, \cite{AlfImTrans}, \cite{AlfMedian}.
\end{remark} 

\section{Quasi-linear maps} \label{SectQLM}

\begin{definition} \label{qhDef}
Suppose $X$ is LC and $E$ is a vector space over $\mathbb{R}$ equipped with a norm or an extended norm. 
A map $\theta: C_0(X) \longrightarrow E $ is called signed quasi-linear if  it is
linear on each subalgebra $B(f)$, $f \in C_0(X)$. 
If $E$ is ordered we say that $ \theta$ is positive if $ \theta(f) \ge 0$ for each $ f \ge 0$.
A signed quasi-linear map which is positive is called a quasi-linear map or a positive quasi-linear map;
it is called a quasi-homomorphism if it is also multiplicative on each subalgebra $B(f)$, $f \in C_0(X)$.
For an ordered $E$, a map $\theta$ is called conic quasi-linear
if (i) $0 \le g \le f$ implies $\theta(g) \le \theta(f)$;  (ii) $f\, g=0, \, f, g  \ge 0$ implies $ \theta(f+ g) = \theta(f) + \theta(g)$; 
(iii) $\theta(a g + bh) = a \theta(g) + b \theta(h)$ for $g,h \in A^+(f), \ a,b \ge 0$ (for each $f \in C_0(X)$). 
A conic quasi-linear map is called a conic quasi-homomorphism if it is also multiplicative on each cone $A^+(f)$.
When $E= C_0(Y), Y$ is LC, 
a quasi-linear (respectively, a conic quasi-linear) map $\theta$ is called strong if $f\, g=0, \, f, g  \in C_0^+(X)$ implies $ \theta(f) \theta(g) = 0$, and 
 $\theta (B(f) ) \subseteq B( \theta(f))$ (resp.,  $\theta (A^+(f) ) \subseteq A^+( \theta(f))$ for each $f \in C_0(X)$).
 \end{definition} 

As $E$ we may consider $ \mathbb{R}, C_0(Y), C_b (Y), C(Y)$  for a LC space $Y$.
When $E = \mathbb{R}$, a signed quasi-linear (conic quasi-linear, quasi-linear) map is a signed quasi-linear functional 
(p-conic quasi-linear functional, quasi-linear functional).
(Conic) quasi-linear maps produce (p-conic) quasi-linear functionals in a couple of ways, see Lemma \ref{ThetElPro}\ref{1inducedQLF} and
Lemma \ref{QLMcompo} below.

\begin{definition} \label{normthet}
For a (conic) quasi-linear map  $\theta$ 
we define $ \| \theta \| = \sup  \{   \| \theta(f) \| : \  0 \le f \le 1\}.    $   
If $\| \theta \| < \infty$ we say that $ \theta$ is bounded. 
\end{definition}

\begin{remark} \label{Qenorm}
For conic quasi-linear maps,  $ \| \theta \| $ satisfies the following properties: $ \| \alpha \theta\|  = \alpha  \| \theta \|  $ for $ \alpha > 0$, 
$ \| \theta \|  = 0$ if $ \theta = 0$, and $ \| \theta+ \eta \|  \le  \| \theta \|  +  \| \eta \|$. 
For quasi-linear maps, $ \| \theta \| = \sup  \{  \| \theta(f) \| : \  0 \le f \le 1\} =  \sup  \{ \| \theta(f) \|: \  \| f \| \le 1 \} $ is an extended norm; see the next lemma. 
\end{remark}

\begin{lemma} \label{ThetElPro}
Suppose $\theta: C_0(X) \longrightarrow C(Y)$ is a (conic) quasi-linear map, $X, Y $ are LC.
\begin{enumerate}[leftmargin=0.25in, label=(\alph*),ref=(\alph*)]
\item  
$\theta(0) = 0$.
\item \label{1inducedQLF}
For  $y \in Y$ the map $\rho: C_0(X) \longrightarrow \mathbb{R}$ given by $\rho(f) = \theta(f)(y)$ is a real-valued (p-conic) quasi-linear functional.
If $ \theta$ is a (conic) quasi-homomorphism, then $\rho$ is simple, and $ \| \theta \| \le 1$.
\item \label{qlmADD}
Suppose  $\theta$ is a quasi-linear map. If $ f g = 0, \, f, g \in C_0(X)$ then  
$\theta(f+ g) = \theta(f) + \theta(g)$. In particular, $\theta(f)  = \theta(f^+) - \theta(f^-).$
\item \label{thnorm2de}
 $\| \theta(f) \| \le \max \{ \| \theta(f^+)\|,  \| \theta(f^-) \| \}$  and 
$ \| \theta \| = \sup  \{  \| \theta(f) \| : \  0 \le f \le 1\} =  \sup  \{ \| \theta(f) \|: \  \| f \| \le 1 \}   $ for a quasi-linear map $ \theta$.
\item \label{qlmMON}
A quasi-linear map $\theta$  is monotone on $C_c(X)$, i.e. $ f \le g, f, g \in C_c(X)  \Longrightarrow  \theta(f) \le \theta(g)$.
\item \label{qlmMONb}
A bounded quasi-linear map $\theta$  is monotone, i.e. $ f \le g, f, g \in C_0(X)  \Longrightarrow  \theta(f) \le \theta(g)$.
\end{enumerate}
\end{lemma}

\begin{proof}
In part \ref{1inducedQLF} for a (conic) quasi-homomorphism, by Theorem \ref{rhosimple} $ \rho$ is simple, 
and then by Remark \ref{RemBRT}\ref{RHOsvva} $ | \theta(f) (y) | \le 1$ for each $y$
and each $ 0 \le f \le 1$. 
For parts  \ref{qlmADD}, \ref{qlmMON} and  \ref{qlmMONb}, for each $y \in Y$ apply \cite[L. 2.2(q2), L. 2.2(iii), L. 4.2]{Butler:QLFLC} to
quasi-linear functional $\rho(f) = \theta(f)(y)$.
The rest is easy. 
\end{proof} 

\noindent
Now, for $f \in C_0(X)$ consider the map $\tilde f:  \mathbf{DTM}(X)  \longrightarrow \mathbb{R}$ given by $ \tilde f (\mu) = \int_X f \, d\mu$. 

\begin{remark} \label{fhatLin}
Note that $\tilde f $ is continuous on $ \mathbf{DTM}(X)$.  Also,  $ \tilde f (\mu + \nu) = \tilde f(\mu) + \tilde f(\nu)$ and 
$\tilde f (t \mu) = t \tilde f (\mu)$ for $ t \ge 0$.
\end{remark}

The next result gives an example of a quasi-linear map and a conic quasi-linear map. 

\begin{proposition} \label{Psifhat}

(I) Let $\Psi: C_0(X) \longrightarrow C( \mathbf{TM}(X) )$ be given by $\Psi(f) = \tilde f$. 
Then $\Psi$ is an order preserving injective quasi-linear map.
If   $\tilde f $ is defined on $X^\sharp$ then $\Psi: C_0(X) \longrightarrow C(X^\sharp)$ is an isometric quasi-homomorphism, and 
$\Psi ( C_0(X))$ is closed. 
(II) If $\Psi: C_0(X) \longrightarrow C( \mathbf{DTM}(X))$, $\Psi(f) = \tilde f$, 
then $\Psi$ is an order preserving injective conic quasi-linear map.
If   $\tilde f $ is defined on $X^{\flat}$ then  $\Psi: C_0(X) \longrightarrow C(X^{\flat})$ is an isometric conic quasi-homomorphism, and 
 $\Psi (C_0(X))$ is closed. 
\end{proposition}

\begin{proof}
(I).
If $ f \le g, \ f,g \in C_0(X)$ then by Remark \ref{RemBRT}\ref{RHOsvva} $\int f \, d\mu \le \int g \, d \mu $,
i.e. $ \tilde f (\mu) \le \tilde g (\mu) $  for any $ \mu$, i.e. $ \tilde f \le \tilde g$,
and $ \Psi$ is order-preserving. 
If $f(x) \neq g(x)$ then $\tilde{f} (\delta_x)  \ne \tilde{g} (\delta_x)$, so $\Psi$ is injective. 
Let $f,g \in B(h)$ for  $ h \in C_0(X)$. For any  $ \mu \in \mathbf{TM}(X) $
\begin{align*} 
\Psi(f + g) (\mu) &= \int_X (f+g) \, d\mu = \int_X f \, d\mu + \int_X g \, d\mu 
= \tilde f (\mu) + \tilde g (\mu)  \\ 
&= \Psi(f) (\mu) + \Psi(g)(\mu),
\end{align*}
so $ \Psi(f + g)  = \Psi(f) +  \Psi(g)$. Then $\Psi$ is linear on each subalgebra. 
If $ \mu $ is simple,  by Theorem \ref{rhosimple} 
$$ \Psi(fg) (\mu) = \int_X fg \, d\mu =( \int_X f\, d\mu)(\int_X g \, d\mu) = \Psi(f)( \mu) \Psi(g) (\mu),$$
i.e. $ \Psi$ is also multiplicative on each subalgebra. So $\Psi$ is a quasi-homomorphism. 

For a simple $ \mu$, by Remark \ref{RemBRT}\ref{RHOsvva} $| \int f \, d\mu | \le \| f \|$, 
so $ \| \tilde{f} \| = \sup \{| \int f \, d\mu | :   \mu \mbox {   is simple  } \} \le \| f\|$. Let $x$ be the point of maximum of $|f|$; then 
 $ \| \tilde{f}  \|  \ge | \tilde{f}( \delta_x)  | = |f(x)| = \| f \|$. Thus,  $\Psi$ is isometric. 
 Since $C_0(X)$ is complete,  $\Psi ( C_0(X))$ is closed.  
 
(II). Is proved similarly, noting that integration with respect to a finite deficient topological measure is a p-conic quasi-linear functional.  
 \end{proof}
 
\begin{remark}
Integration with respect to a (deficient)  topological measure $\mu$ is not linear in general,
i.e. $\int_X (f+g) \, d\mu \neq \int_X f\, d\mu + \int_X g\, d\mu$ for some $f,g \in C_0(X)$. Then $\Psi$ is not linear: 
$\Psi(f+g) (\mu) \neq \Psi(f) (\mu) + \Psi(g) (\mu)$, so $\Psi(f+g) \neq \Psi(f) + \Psi(g)$.  In \cite{Aarnes:Pure} $\Psi$ is called a nonlinear Gelfand transform.
\end{remark}

\begin{definition}
A (deficient) topological measure is called representable if it is in the closed convex hull of simple (deficient) topological measures. 
A quasi-integral, corresponding to a representable (deficient) topological measure is called representable.
\end{definition}

\noindent
As we shall see in Section \ref{ITqhCorr}, $\Psi$ plays a special role for representable quasi-linear functionals. 

\begin{theorem} \label{normbddOP}
Suppose $X$ is a LC space,  $E$ is a normed space, and $ \theta: C_0(X) \longrightarrow E$  
is a (nonlinear) operator
such that $\theta (af) = a \theta(f)$ for any $a \ge 0$, and 
$0 \le f \le g$ implies $ \| \theta(f ) \| \le \| \theta(g) \|$. Then 
$\| \theta \| = \sup  \{  \| \theta(f) \| : \  0 \le f \le 1\} < \infty$.  
\end{theorem}

\begin{proof}
The statement can be obtained by an easy adaptation of an argument in \cite[L. 2.3]{Alf:ReprTh}.
\end{proof}

\begin{corollary} \label{qlfbd}
Suppose $X, Y$ are LC. A conic quasi-linear map $ \theta: C_0(X) \longrightarrow C_b(Y)$ is bounded;
for a compact $X$ a quasi-linear map $ \theta: C(X) \longrightarrow C_b(Y)$ is monotone and  bounded. 
A  quasi-homomorphism $ \theta: C_0(X) \longrightarrow C_b(Y)$ is bounded and monotone.
\end{corollary}

\begin{proof}
When $X$ is compact, any quasi-linear functional is monotone (see  \cite[L. 4.2]{Butler:QLFLC} or \cite[L. 4.1]{Aarnes:TheFirstPaper}),  
and for a quasi-linear map  $ \theta: C_0(X) \longrightarrow C_b(Y)$  considering  quasi-linear functionals $ \rho(f) = \theta(f)(y)$ we have:
$f \le g \Longrightarrow \theta(f)(y) \le \theta (g) (y)$, so $ \| \theta(f) \| \le \| \theta(g) \|$.
Apply Theorem \ref{normbddOP} and  Lemma \ref{ThetElPro}.
\end{proof}

$f_{\alpha}  \searrow f$ means that a decreasing net of functions converges pointwise to $f$;  $f_{\alpha}  \nearrow f$ is defined similarly. 

\begin{theorem} \label{INTsmoo}
Let $ \mu$ be a finite deficient topological measure on a LC space $X$. 
\begin{enumerate}[leftmargin=0.25in, label=(\Roman*),ref=(\Roman*)]
\item
Suppose $f_{\alpha}  \nearrow f,  \ f_{\alpha}, f \in C_0(X)$. Then $ \int_X f_{\alpha} \, d\mu \nearrow  \int_X f \, d\mu$. 
\item 
For a finite topological measure also  $f_{\alpha}  \searrow f  \ f_{\alpha}, f \in C_0(X)$ implies $ \int_X f_{\alpha} \, d\mu \searrow  \int_X f \, d\mu$.
\end{enumerate} 
 \end{theorem}
 
\begin{proof}
(I).
If $f_{\alpha}  \nearrow f,  \ f_{\alpha}, f \in C_0(X)$ then sets $f_{\alpha}^{-1}((t, \infty)) \nearrow f^{-1}((t, \infty))$  for each $t >0$.
$ \mu$ is $\sigma$-smooth on open sets, so $R_{1,\mu, f_{\alpha}} (t) \nearrow R_{1,\mu, f} (t).$  
By \cite[L. 6.3]{Butler:ReprDTM} each $R_1(t)$ is real-valued, right-continuous, nonincreasing, so it is lower semicontinuous. 
By \cite[L.7.2.6]{Bogachev} $\int_a^b R_{1,\mu, f_{\alpha}} (t) \, dt \nearrow \int_a^b R_{1,\mu, f} (t) \, dt $ (where 
$[a,b]$ contains $f(X)$), so by formula (\ref{rfform}) $ \int_X f_{\alpha} \, d\mu \nearrow  \int_X f \, d\mu$.  \\
(II).
Suppose $f_{\alpha}  \searrow f$, where  $f_{\alpha}, f \in C_0(X)$.  
By formula (\ref{intsplit}) and the first part we may assume that  $f\ge 0$. 
Since $\mu$ is $\sigma$-smooth on compact sets, we see that 
for $R_{2, \mu, f_{\alpha}} (t) \searrow  R_{2, \mu, f} (t)$ for every $t >0$. 
By \cite[L. 3.2]{GrubbLaBerge:Smooth} and formula (\ref{rfformp})  $ \int_X f_{\alpha} \, d\mu \rightarrow  \int_X f \, d\mu$.    
\end{proof}

\begin{Theorem} \label{wEq}
Let $X, Y$ be LC.
\begin{enumerate}[leftmargin=0.25in, label=(\roman*),ref=(\roman*)]
\item \label{wEqi}
Suppose $ \theta: C_0(X) \longrightarrow C(Y)$ is a quasi-linear map
such that $ \sup  \{ |\theta(f)(y)| :  \| f \| \le 1 \} = N_y < \infty$ for each $y$. 
Then there exists a continuous map $w: Y \longrightarrow \mathbf{TM}(X)$) such that
$\theta(f) (y) = \int_X f \, dw_y\, $  for every $f \in C_0(X)$, and $w_y(X) \le N_y$. Here $w_y = w(y)$.
If $f_{\alpha}  \nearrow f,  \ f_{\alpha}, f \in C_0(X)$ then $ \theta(f_{\alpha}) \nearrow \theta(f) $;  
if $f_{\alpha}  \searrow f,  \ f_{\alpha}, f \in C_0(X)$ then $ \theta(f_{\alpha}) \searrow \theta(f) $.

\item \label{wEqii}
If $ \theta: C_0(X) \longrightarrow C(Y)$ is a conic quasi-linear (respectively, a bounded quasi-linear) map,
then there exists a continuous map  $w: Y \longrightarrow  \mathcal M $ such that 
$\theta(f) (y) = \int_X f \, dw_y\, $  for every $f \in C_0^+(X)$ (resp., $f \in C_0(X)$).
Here $ \mathcal M = \{ \mu: \mu(X) \le \| \theta\| \} $  is a uniformly bounded in variation family of deficient topological (resp., topological) measures on $X$.
For a conic quasi-linear map, if $f_{\alpha}  \nearrow f,  \ f_{\alpha}, f \in C_0^+(X)$,
then $ \theta(f_{\alpha}) \nearrow \theta(f) $. 

\item  \label{wEqiii}
If $ \theta$ is a conic quasi-homomorphism (resp.,  quasi-homomorphism) 
then  $w: Y \longrightarrow X^\flat$ (resp.,  $w: Y \longrightarrow X^{\sharp}$), 
and if $ \theta: C_0(X) \longrightarrow C_0(Y)$, then $w$ is a continuous k-proper map. 
\end{enumerate}
\begin{enumerate}[leftmargin=0.25in, label=(\Roman*),ref=(\Roman*)]

\item \label{wEqI}
To each continuous map $w: Y \longrightarrow \mathbf{TM}(X)$ (resp.,  $w: Y  \longrightarrow  \mathbf{DTM}(X)$)  
corresponds an order-preserving quasi-linear (resp., conic quasi-linear) map $ \theta: C_0(X) \longrightarrow C(Y)$ such that
$\theta(f) (y) = \int_X f \, dw_y$, and $ \sup  \{ |\theta(f)(y)| :  \| f \| \le 1 \} \le w_y$ for each $y$. 

\item \label{wEqII}
If  $w: Y \longrightarrow  \mathcal M$, where $w$ is a continuous proper funciton and   $ \mathcal M = \{ \mu : \mu(X) \le N\} $ 
is a uniformly bounded in variation family of (deficient) topological measures, then  $ \| \theta \| \le N$, $ \theta(f) \in C_b(Y)$,  and  
$ \theta (C_0^+(X)) \subseteq C_0^+(Y)$. Hence, $ \theta (C_0(X)) \subseteq C_0(Y)$ for a quasi-linear map $ \theta$. 
If $ f, g \in C_c^+(X)$ then $ \| \theta(f) - \theta(g) \| \le N \| f -g\|$; and if $X$ is compact, then 
$ \| \theta(f) - \theta(g) \| \le N \| f -g\|$ for any $f,g \in C(X)$.
 A quasi-linear $ \theta$ is a Lipschitz map: for  $f, g \in C_0(X)$,
$\| \theta(f) - \theta(g) \| \le 2 N \| f -g\|$. 

\item \label{wEqIII}
If $w: Y \longrightarrow X^\sharp$ (resp.,  $w: Y \longrightarrow X^{\flat}$)  is a continuous k-proper funciton, then 
$\theta$ is a quasi-homomorphism (resp., a conic quasi-homomorphism), $ \| \theta \| \le 1$, $ \theta(C_0^+(X)) \subseteq C_0^+(Y)$, 
and $\theta(C_c^+(X)) \subseteq C_c^+(Y)$. 
Hence, $ \theta (C_0(X)) \subseteq C_0(Y)$ and  $ \theta (C_c(X)) \subseteq C_c(Y)$ for a quasi-linear map $ \theta$. \\
\end{enumerate}
In particular, there is a 1-1 correspondence between quasi-linear maps $ \theta: C_0(X) \longrightarrow C(Y)$ from part  \ref{wEqi} and  
continuous functions $w: Y \longrightarrow \mathbf{TM}(X)$.
There is a 1-1 correspondence between bounded quasi-linear maps $ \theta: C_0(X) \longrightarrow C_b(Y)$ 
(resp., conic quasi-linear maps $ \theta: C_0^+(X) \longrightarrow C_b^+(Y)$) and continuous
functions $w: Y \longrightarrow  \mathcal M$, where $\mathcal M$ is a uniformly bounded in variation family of topological (resp., deficient topological)  measures.
There is a 1-1 correspondence between quasi-homomorphisms $ \theta: C_0(X) \longrightarrow C_0(Y)$   
(resp.,  conic quasi-homomorphisms  $ \theta: C_0^+(X) \longrightarrow C_0^+(Y)$)  and 
continuous k-proper functions $w: Y \longrightarrow X^\sharp $ (resp.,   $w: Y \longrightarrow  X^{\flat}$).
In these correspondences $ \theta$ has properties stated in \ref{wEqi},  \ref{wEqii}, \ref{wEqII} and \ref{wEqIII}.
\end{Theorem}

\begin{proof}
Let $y \in Y$. Consider the map  $\rho : C_0(X) \longrightarrow \mathbb{R}$ given by $\rho(f) = \theta(f)(y)$.  \\
\ref{wEqi}. By assumption  $\rho$ is bounded, so by Remark \ref{RemBRT}
$  \theta(f)(y) = \int_X f \, dw_y$ 
for some finite topological measure $w_y$ on $X$. 
Define    $w: Y \longrightarrow \mathbf{TM}(X)  $ by $w(y) = w_y$. 
Suppose a net $y_{\alpha} \longrightarrow y$. 
For each $f \in C_0^+(X)$ 
by continuity of $\theta(f)$, $\theta(f) (y_{\alpha}) \longrightarrow \theta(f) (y)$, i.e. 
$\int_X f \, dw_{y_{\alpha}} \longrightarrow \int_X f\, dw_y$, so $ w(y_{\alpha}) \Longrightarrow w(y)$, and $w$ is continuous. 
By Remark \ref{RemBRT}\ref{mrDTM}, $w_y(X) = \sup \{ \rho(f): f \in C_c(X), 0 \le f \le 1 \} \le N_y$. 
For the last two statements apply Theorem \ref{INTsmoo} to $\rho(f)$, and note that an increasing/decreasing net $(f_{\alpha})$ will give 
an increasing/decreasing net $(\theta(f_{\alpha}))$ by monotonicity of $ \rho$ in Remark \ref{RemBRT}\ref{RHOsvva}. \\
\ref{wEqii}. Follows from part \ref{wEqi}; note that by Corollary \ref{qlfbd} a conic quasi-linear map is bounded. \\
\ref{wEqiii}.  Since $\theta$ is multipilicative on the cone (resp., on the algebra) generated by $f$,  we have:
$\rho(f^2) = \theta (f^2)(y) = (\theta(f)(y))^2 = (\rho(f))^2$. By Theorem \ref{rhosimple}
$\rho$ is a simple  p-conic quasi-linear (resp., simple quasi-linear) functional.  
Thus, in $  \theta(f)(y) = \int_X f \, dw_y$,  $w_y$ is a simple deficient topological measure (resp., a simple topological measure).

Suppose $K \in \mathscr{K}(X)$. If $ y \in w^{-1}(K^\sharp)$ (or $ y \in w^{-1}(K^{\flat})$) then $w(y) (K) = 1$, so
for a function $g \in C_c^+(X)$ such that $g=1$ on $K$, by Remark \ref{RemBRT}\ref{mrDTM}  we have
$ \theta(g) (y) = \int_X  g\, dw_y \ge w_y( K) = 1$. Then $ y \in (\theta(g))^{-1}( [1, \infty))$, a compact set. Thus, $w$ is k-proper.  \\
\ref{wEqI}.  Suppose $w: Y \longrightarrow \mathbf{TM}(X)$ (respectively,  $w: Y  \longrightarrow   \mathbf{DTM}(X)$)  is continuous. 
For  $f \in C_0(X)$ define a function  $\theta(f)$ on $Y$ by $\theta(f) (y) = \int_X f \, dw_y.$
Then $\theta(f)$ is continuous, since $y_{\alpha} \longrightarrow y$ implies $w_{y_{\alpha}} \Longrightarrow   w_y$, so 
$\int_X f \, dw_{y_{\alpha}} \longrightarrow \int_X f\, dw_y$, i.e. $\theta(f) (y_{\alpha}) \longrightarrow \theta(f)(y)$. 
If $ f \le g$ then for each $y$ by Remark \ref{RemBRT}\ref{RHOsvva}  $\, \theta(f)(y) \le \theta(g)(y)$, and $ \theta$ is order-preserving.
Suppose $g, h \in B(f)$ (resp., $g, h \in A^+(f))$, $ f \in C_0(X)$. 
Since integration with respect to a finite topological (resp., finite deficient topological) measure is a quasi-linear functional 
(resp., a p-conic quasi-linear functional),  we see that for each $y$,
$\theta(g+ h)(y) = \theta(g)(y) + \theta(h) (y).$
Thus, $\theta(g + h) = \theta(g) + \theta(h)$. 
Then $\theta$ preserves linear combinations on $B(f)$ (resp., nonnegative linear combinations on $A^+(f)$) for $ f \in C_0(X)$.
If $w_y$ is a deficient topological measure, then property (ii) in Definition \ref{qhDef} holds for $\theta$ because 
it holds for each $ \theta(f)(y)$  by \cite[L. 7.7]{Butler:ReprDTM}. Thus, $ \theta$ is a (conic) quasi-linear map.
By Remark \ref{RemBRT}\ref{RHOsvva}, $ \sup  \{ |\theta(f)(y)| :  \| f \| \le 1 \} \le w_y$. \\
\ref{wEqII}. 
Suppose  $w: Y \longrightarrow  \mathcal M $ is proper and  $ \mathcal M = \{ \mu : \mu(X) \le N\} $. 
For any $ 0 \le f \le 1$ by formula (\ref{estin}),  $0 \le \theta(f)(y) \le w_y(X) \le N$ for each $y$. Thus, $ \| \theta \| \le N$.  
Let $ f \in C_0^{+}(X)$, $ 0 \le f \le 1$. $\theta(f) \ge 0$, and  to show that $ \theta(f) \in C_0^+(Y)$ we consider
$ \theta(f) (y) = \int_0^1 R_{2, w_y, f} (t) dt \ge 2c$ where $0 < 2c \le w_y(X)$ (see Remark \ref{RemBRT}).
The function $ R_{2, w_y, f} (t)$ is nonincreasing, and we let $ t_y = \sup\{t:    R_{2, w_y, f} (t) \ge c \}$. Note that 
$ t_y >0$ (otherwise $\theta(f)(y) \le c$), and $t_y \le 1$ (since $ R_{2, w_y, f} (t) = 0 $ for $t >1$).
Then  $ 2c \le \int_0^{t_y}  R_{2, w_y, f} (t) dt  + \int_{t_y}^1 R_{2, w_y, f} (t) dt  \le w_y(X) t_y + c(1 - t_y) < w_y(X) t_y + c$.
Then $w_y(X) t_y > c$, so $ t_y > \frac{c}{w_y(X)} \ge \frac cN$,
and $ R_{2, w_y, f} (\frac cN)  \ge c$, i.e. $ w_y(K) \ge c$ for $K  = f^{-1}([\frac cN, \infty))$.  
Thus, $w(y) \in D = \{ \mu \in  \mathcal M: \mu(K) \ge c \}$. The set $D$ is closed (see \cite[T. 2.1]{Butler:WkConv}), hence, compact. 
Then $ y \in w^{-1}(D)$, a compact set because $w$ is proper. 
We see that $ \theta (C_0^+(X)) \subseteq C_0^+(Y)$. 
For a quasi-linear map, $ \theta(f) = \theta(f^+) - \theta(f^-)$, so $ \theta (C_0(X)) \subseteq C_0(Y)$. 
For any $y$ and $f, g \in C_c^+(X)$  by \cite[L. 7.12(z7)]{Butler:ReprDTM} we have $|\theta(f)(y) - \theta(g)(y)| \le N \| f - g \|$, so
$\| \theta(f) - \theta(g) \| \le N \| f - g \|$. The remaining assertions follows similarly from  Remark \ref{RemBRT}\ref{RHOsvva}. \\
\ref{wEqIII}.   
We shall show that $\theta(f) \in C_0^+(Y)$ for $ f \in C_0(X), 0 \le f \le 1$.
Suppose $\theta(f)(y)  = \int_0^1 R_{2, w_y, f}(t) \, dt \ge 2c$, where $0 < 2c \le 1$. 
The function $R_{2, w_y, f}(t) $ is a step function that assumes only values $1$ and $0$ and has single discontinuity (at a point $b_y \ge 2c$).
Then $R_{2, w_y, f}(c) = 1$, i.e. $w_y(f^{-1}([c, \infty)))= 1$, i.e. $y \in w^{-1}((f^{-1}([c, \infty)))^\flat)$, 
a compact set, since $w$ is k-proper. Thus, $\theta(f) \in C_0^+(Y)$.

Let $f \in  C_c^+(X)$. 
If $y \notin  w^{-1}((supp f)^\flat) $ then $w_y \notin (supp f)^\flat$, i.e. $w_y (supp f) = 0$. 
Then $w_y(Coz f) = 0$, which by \cite[Th. 49]{Butler:Integration} 
means that $ \int f \, dw_y =0$, i.e. $\theta(f)(y) = 0$. Thus, $supp \,  \theta(f) \subseteq w^{-1}((supp f)^\flat)$, so $\theta(f)$ is compactly supported, since 
$w$ is k-proper. 

Suppose $g, h \in B(f)$ (resp.,  $g, h \in A^+(f)$), and  $y \in Y$.
Since $w_y$ is simple, by Theorem \ref{rhosimple}
$ \theta(gh)(y)  =( \theta(g) (y)) ( \theta(h)(y)), $
 so $ \theta(gh) = \theta(g) \theta(h).$ 
Thus, $ \theta$ is a (conic) quasi-homomorphism.
\end{proof}

\begin{remark} \label{ptBdwy}  
When $X$ and $Y$ are compact, 
the correspondence between bounded quasi-linear maps $\theta$ and continuous functions 
$w: Y \longrightarrow  \mathcal M$ was stated (without proof) in \cite[Prop. 9]{Grubb:Signed}, and the correspondence between quasi-homomorphisms and
functions $w$ first appeared in \cite{Aarnes:ITfirst}.
\end{remark} 
 
Although $\theta$ is not a linear operator, we still have one of the most important results for linear operators hold for $\theta$ as well.

\begin{theorem} \label{bddiscont}
Let $X, Y$ be LC spaces.
A continuous conic quasi-linear map $ \theta: C_0(X) \longrightarrow E$ is bounded.   
A quasi-linear map $ \theta: C_0(X) \longrightarrow C(Y)$ is continuous iff $ \theta$ is bounded. 
Any quasi-homomorphism $ \theta: C_0(X) \longrightarrow C(Y)$ is bounded, continuous, and monotone.
\end{theorem}

\begin{proof}
Let  $ \theta: C_0(X) \longrightarrow E$ be a (conic) quasi-linear map. 
By the standard argument, if $\theta$ is continuous, then $\theta$ is bounded. 
Now let $ \theta: C_0(X) \longrightarrow C(Y)$ be a bounded quasi-linear map.
Let $f, g \in C_0(X)$. 
By Theorem \ref{wEq}  $\| \theta(f) - \theta(g) \| \le 2 \| f - g \| \| \theta \|, $ so $ \theta $ is continuous.  
For the last statement use Corollary \ref{qlfbd}.  
\end{proof}

\noindent
As  the next theorem shows, it is enough to define bounded quasi-linear maps on $C_c(X)$ instead of $C_0(X)$. 

\begin{theorem}  \label{extfromCcOp}
Suppose $X,Y$ are LC,  $E$ is $\mathbb{R}$, $C_b(Y)$ or $C_0(Y)$.
A bounded quasi-linear map $ \theta': C_c(X) \longrightarrow E$ 
extends uniquely to a bounded and continuous quasi-linear map $ \theta: C_0(X) \longrightarrow E$ with the same norm.
A bounded quasi-linear map  $ \theta: C_0(X) \longrightarrow E$
is the unique extension of a bounded  quasi-linear map $ \theta': C_c(X) \longrightarrow E$ ($\theta'$ is the restriction of $\theta$  to $C_c(X)$).
\end{theorem}

\begin{proof}
For $E=\mathbb{R}$ this is \cite[Th. 4.8]{Butler:QLFLC}. For $E = C_0(Y)$ or $C_b(Y)$
one can use almost verbatim the proof of \cite[Th. 4.8]{Butler:QLFLC}, replacing $ \rho$ with $\theta'$ and $L $ with $ \theta(f)$.
\end{proof} 

\begin{lemma}  \label{thetBfSe}
Suppose $X$ and $Y$ are LC,  $ \theta: C_c(X) \longrightarrow C_0(Y)$ is a quasi-linear map, and $ f \in C_c(X)$.
Let $\phi $ be continuous on  $\overline{f(X)} \cup \overline{\theta(f)}$ ($  \phi(0) = 0$ if $X$ is noncompact). 
Then  $ \theta(\phi (f)) = \phi (\theta(f))$, and $ \theta (B(f))  \subseteq B(\theta(f))$.
\end{lemma}

\begin{proof} 
Choose $g \in C_c^+(X), \, g \le 1$ such that $ g=1$ on $supp f$.
For $ \epsilon >0$  let $p_n$ be a polynomial on compact $\overline{f(X)} \cup \overline{\theta(f)}$  ($p_n(0)=0$ for a noncompact $X$) 
such that 
$ \|\phi -  p_n  \|, \| \theta(g) \|  \|\phi -  p_n  \| < \epsilon$.
Clearly, $p_n \circ f, \phi \circ f$ are compactly supported. 
Let $y \in Y$. With notations $ \rho(f) = \theta (f) (y)$,  using \cite[formula (3.1)]{Butler:ReprDTM}  we have
\begin{align*}
|\theta(\phi \circ f)(y)  &-  \theta(p_n \circ f) (y)  |  = | \rho(\phi \circ f) - \rho (p_n \circ f) |  \le   \|\phi \circ f   -  p_n \circ f  \|   \rho(g)\\
&=   \| \phi \circ f - p_n \circ f  \| \theta(g)(y) \le   \| \phi \circ f - p_n \circ f  \|  \| \theta(g)\| < \epsilon.
 \end{align*}
So $\|  \theta(\phi ( f)) -  \theta(p_n( f)) \| \le \epsilon$. By linearity of $ \theta$ on subalgebras $ \theta(p_n ( f)) = p_n (\theta(f))$, so 
$\| \theta(\phi(f)) - \phi(\theta(f)) \| \le \|  \theta(\phi(f)) - \theta(p_n (f)) \| + \|  p_n (\theta(f)) -  \phi(\theta(f)) \| < 2 \epsilon$. 
Thus,  $ \theta(\phi (f)) = \phi (\theta(f))$.
\end{proof}

\begin{remark} \label{ThetMul}
Suppose $B_1, B_2$ are Banach algebras, $ \theta: B_1 \longrightarrow  B_2$ is a linear map such that $ \theta(f^2) = (\theta(f))^2$ for 
any $ f \in B_1$. Then $ \theta$ is multiplicative, i.e. $ \theta (fg) = \theta(f) \theta(g)$. This follows from equality of
$\theta( (f+g)^2)$ and $ (\theta(f+g))^2$. Likewise, if $ \theta$ is a (conic) quasi-linear map such that $ \theta(f^2) = (\theta(f))^2$ 
for any $ f \in C_0(X)$ then $ \theta$ is multiplicative on each subalgebra $B(h)$ (resp., on each cone $A^+(h)$), i.e. 
$ \theta$ is a (conic) quasi-homomorphism. The next theorem gives criteria
for a (conic) quasi-linear map to be a (conic) quasi-homomorphism.
\end{remark}

\begin{theorem} \label{qhEqDef}
Let $X$, $Y$ be LC spaces, 
Suppose $\theta: C_0(X) \longrightarrow C_0(Y)$ is a quasi-linear map (respectively, a conic quasi-linear map).
TFAE:
\begin{enumerate}[leftmargin=0.25in, label=(\arabic*),ref=(\arabic*)]
\item \label{qhontoB1}
$\theta(\phi (f)) = \phi( \theta(f))$ for any $ f \in C_0(X)$ and any $ \phi  \in C(K)$, $ \phi(0) = 0$ if $X$ is noncompact; (resp., $ \phi$ is
also nondecreasing). So $\theta (B(f)) = B(\theta(f))$  (resp., $ \theta (A^+(f)) = A^+(\theta(f))$). 
Here  $K$ is any compact containing $ \overline{f(X)}$ and  $\overline{\theta(f)}$. 
\item \label{qhontoB2}
$\theta(\phi (f)) = \phi( \theta(f))$  for any $ f \in C_0(X)$ and any absolutely monotone $ \phi  \in C(K)$, $ \phi(0) = 0$ if $X$ is noncompact.
Here  $K$ is any compact containing $ \overline{f(X)}$ and  $\overline{\theta(f)}$. 
\item \label{qhontoB3}
$ \theta(f^2) = (\theta(f))^2$ for any $ f \in C_0(X)$.
\item \label{qhontoB4}
 $ \theta$ is multiplicative on each subalgebra $B(f)$, i.e. $ \theta$ is a quasi-homomorphism; (resp., on each cone $A^+(f)$, i.e. 
 $ \theta$ is a conic quasi-homomorphism).
\item \label{qhontoB5}
$\theta(p (f)) = p( \theta(f))$ for any $ f \in C_0(X)$ and any polynomial $p$ on $K$, $ p(0) = 0$ if $X$ is noncompact; 
(resp., $p$ has nonnegative coefficients).
Here  $K$ is any compact containing $ \overline{f(X)}$ and  $\overline{\theta(f)}$. 
\end{enumerate}
\end{theorem}

\begin{proof}
\ref{qhontoB1} $\Longrightarrow$  \ref {qhontoB2} $\Longrightarrow$  \ref {qhontoB3},  
\ref{qhontoB4} $\Longrightarrow$  \ref {qhontoB5} $\Longrightarrow$  \ref {qhontoB3}   
are clear (see \cite[p.154]{Widder}), 
and \ref{qhontoB3} $\Longrightarrow$  \ref{qhontoB4} by Remark \ref{ThetMul}.  
\ref{qhontoB4} $\Longrightarrow$  \ref{qhontoB1} since it holds pointwise applying Theorem \ref{rhosimple} to 
$ \rho(f) = \theta(f)(y)$ for each $y$.
\end{proof}

\begin{lemma} \label{rhofg0mult}
Suppose $ f, g \in C_0^+(X), f\, g = 0$, $X,Y$ are LC.  If $ \rho$ is a simple quasi-integral on $C_0(X)$  then $ \rho(f) \rho(g) = 0$. 
If  $\theta: C_0(X) \longrightarrow C(Y)$ is a conic quasi-homomorphism, then $ \theta(f) \theta(g) = 0$. 
Hence, a (conic)  quasi-homomorphism $\theta: C_0(X) \longrightarrow C_0(Y)$ is a strong (conic) quasi-linear map.
\end{lemma}

\begin{proof}
Since $\rho$ is simple, $ \rho(f)   = \int 	f \, d\mu$ for some simple (deficient) topological measure  $ \mu$.
Open sets $ U  = Coz \, f$ and $V = Coz \, g$ are disjoint, so we can not have $ \mu(U) = \mu(V) = 1$. 
Say, $ \mu(U) = 0$. Then by \cite[Th. 49]{Butler:Integration} $\rho(f) = 0$, so $ \rho(f) \rho(g) = 0$.   

If $\theta: C_0(X) \longrightarrow C(Y)$ is a conic quasi-homomorphism, then for each $y$ by Theorem \ref{wEq}\ref{wEqiii}
the functional  $ \rho(h) = \theta (h) (y)$ is simple, so $\theta(f)(y) \theta(g) (y) = \rho(f) \rho(g) =0$.
\end{proof} 

\begin{remark} \label{tTheta}
If $ \theta$ is a (conic) quasi-linear map or a strong (conic) quasi-linear map, then so is $t \theta$ for any $t \ge 0$; 
it correspond to functions $ tw$, where $w$ is from Theorem \ref{wEq}.
The sum of two (conic) quasi-linear maps is a (conic) quasi-linear map, so 
(conic) quasi-linear maps constitute a positive cone. 
Let $\theta $ be a (conic) quasi-homomorphism. For any positive $ t \neq 1$ the map $t \theta$ is not multiplicative
on any nonzero singly generated subalgebra (resp., on singly generated cone), 
but satisfies $ (t\theta) (B(f)) \subseteq B(\theta(f))$ (resp.,  $ (t\theta) (A^+(f)) \subseteq A^+(\theta(f))$) for any $ f \in C_0(X)$.  
Thus, $ t \theta$ is a strong (conic) quasi-linear map;  it is not a (conic) quasi-homomorphism for $ t \neq 1$.
\end{remark}
  
\begin{lemma} \label{QLMcompo}
If  $ \theta: C_0(X) \longrightarrow C_0(Y)$  is a strong (conic) quasi-linear map,  
and $ \xi: C_0(Y) \longrightarrow E$  is a (conic) quasi-linear map,
then $ \xi \circ \theta : C_0(X) \longrightarrow E$ is a (conic) quasi-linear map. 
\end{lemma}

\begin{remark}
By Lemma \ref{rhofg0mult} and Lemma \ref{QLMcompo} the composition of two (conic) quasi-homomorphisms is a
(conic) quasi-homomorphism. The composition of two quasi-linear maps need not be a quasi-linear map, see Example \ref{qhOnXY} below.
\end{remark}

\begin{example} \label{2inducedQLF}
Suppose  $\theta: C_0(X) \longrightarrow C_0(Y)$ is a strong conic quasi-linear map or a bounded strong  quasi-linear map (in particular, a (conic) quasi-homomorphism).
Let $ \nu$ be a finite (deficient) topological  measure on $Y$  with corresponding bounded (p-conic)  quasi-linear functional $\xi$.
Let  $\rho: C_0(X) \longrightarrow \mathbb{R}$ given by $\rho(f) = \int_Y \theta(f)\, d \nu =( \xi \circ \theta) (f)$. 
By Lemma \ref{QLMcompo} $ \rho$ is a bounded (p-conic) quasi-linear functional. 
\end{example}

\begin{definition} \label{qh*def}
For a strong conic quasi-linear (respectively,  bounded strong quasi-linear) map $\theta: C_0(X) \longrightarrow C_0(Y)$  
we define its adjoint $\theta^*:  \mathbf{DTM}(Y) \longrightarrow  \mathbf{DTM}(X)$  (resp., $\theta^*: \mathbf{TM}(Y) \longrightarrow  \mathbf{TM}(X)$)
so that for $\nu \in \mathbf{DTM}(Y)$  (resp., $\nu \in  \mathbf{TM}(Y)$),  $\theta^* \nu$ 
is the finite deficient topological (resp., finite topological) measure on $X$ 
corresponding to functional $\rho(f) =  \int_Y \theta(f) \, d\nu = ( \xi \circ \theta) (f)$ in Example \ref{2inducedQLF}.
\end{definition}

Thus,  for a bounded strong quasi-linear map $ \theta$ and  $f \in C_0(X)$ or for a strong conic quasi-linear map  $\theta$ and $f \in C_0^+(X)$
\begin{align} \label{Intdq*}
 \int_Y \theta(f) \, d\nu = \int_X  f \, d \theta^*\nu. 
 \end{align}
 
\begin{remark}
If in Lemma \ref{QLMcompo}  $\theta$ and $ \xi$ are bounded, then $ \xi \circ \theta$ is bounded. 
Let $v: Z \longrightarrow \mathcal{M}_1$ corresponds to $\xi$ and $w: Z \longrightarrow  \mathcal{M}_2$ corresponds to $ \xi \circ \theta$ by Theorem \ref{wEq}.
By (\ref{Intdq*}) and Theorem \ref{wEq} for each $z \in Z$, 
$ \int_X f \, d(\theta^* v_z) = \int_Y \theta(f) \, d v_z = \xi(\theta(f))(y) =  \int_Z f \, d w_z$ for all $f \in C_0^+(X)$. 
By (\ref{EqByInt}), $ \theta^* v_z = w_z$.
\end{remark} 

\begin{lemma} 
Let $X, Y$ be LC. 
(I) For a bounded strong quasi-linear map  $\theta: C_0(X) \longrightarrow C_0(Y) $ 
its adjoint $\theta^*: \mathbf{TM}(Y)  \longrightarrow \mathbf{TM}(X)$ is continuous. 
If $\theta$ is a quasi-homomorphism then $\theta^*: Y^\sharp \longrightarrow X^\sharp$ is continuous.
(II) For a strong conic quasi-linear map $\theta: C_0(X) \longrightarrow C_0(Y) $  
its adjoint $\theta^*: \mathbf{DTM}(Y)  \longrightarrow  \mathbf{DTM}(X)$ is continuous. 
If $\theta$ is a conic quasi-homomorphism then $\theta^*: Y^{\flat} \longrightarrow X^{\flat}$ is continuous.
\end{lemma} 

\begin{proof}
The proof is similar for both parts, and we shall show it for part (II). 
Let $\nu_{\alpha}  \Longrightarrow \nu, \, \nu_{\alpha}, \nu \in  \mathbf{DTM}(Y)$.  For every $f \in C_0(X)$,
$ \int_Y \theta(f) \, d \nu_{\alpha} \longrightarrow  \int_Y \theta(f) \, d \nu$, i.e.
$\int_X f \, d \theta^* \nu_{\alpha}  \longrightarrow   \int_X f \, d \theta^* \nu$,
so  $\theta^* \nu_{\alpha} \Longrightarrow \theta^* \nu$,   and $ \theta^*$ is continuous.
If $\nu$ is a simple deficient topological measure and $f \in C_0(X)$,  
then by (\ref{Intdq*}) and  Theorem \ref{rhosimple} for a quasi-homomoprhism $ \theta$ we have 
$ \int f^2 \, d \theta^* \nu = \int \theta(f^2) \, d \nu = \int (\theta(f))^2 \, d \nu = \left( \int \theta(f) \, d \nu \right)^2 =  \left( \int f\, d \theta^* \nu \right)^2$,   
so  $\theta^* \nu $ is simple.
\end{proof}

\begin{remark}
We introduce the concept of a conic quasi-linear map. 
When $X$ is compact our definition of a quasi-linear map coincides with the one in \cite{Grubb:qlm}. 
For quasi-homorphisms, our definition is close to one in \cite{AarnesGrubb}. In almost all previous works, 
including \cite{AarnesJohansenRustad}, \cite{Pedersen},   \cite{AlfImTrans},  \cite{Taraldsen:RegQHinLC}, \cite{Taraldsen:ITQMonLC},
a quasi-homomorphism is defined as 
a map from $C(X)$ to $C(Y)$, where $X, Y$ are compact, which is an algebra homomorphism of $B(f)$  onto (into) $B(\theta(f))$ for each $f  \in C(X)$.
In this case each quasi-homomorphism is bounded since $ \| \theta(f) \| \le \| \theta(1) \|$
for $f \in C(X), \| f \| \le 1$. Theorem \ref{qhEqDef} shows that the three different definitions of a quasi-homomorphism 
in all previous papers are equivalent to our definition.   

While working on this paper, the author came up with the idea of a (conic) quasi-linear map as a counterpart for (d-) image transformations, 
but later discovered that the definition of a quasi-linear map in a compact setting was earlier introduced by D. Grubb in  \cite{Grubb:Signed}  and
\cite{Grubb:qlm}. 
\end{remark}
 
\section{Correspondence between (d-) image transformations and (conic) quasi-linear maps} \label{ITqhCorr}

\begin{theorem} \label{ITqh}
Suppose $X$ and $Y$  are LC. There is a 1-1 correspondence between d-image transformations (respectively, image transformations) 
$q$ from $X$ to $Y$ and 
conic quasi-homomorphisms $\theta : C_0^+(X) \longrightarrow C_0^+(Y)$ 
(resp., quasi-homomorphisms $\theta : C_0(X) \longrightarrow C_0(Y)$), 
and there is a 1-1 correspondence between d-image transformations (resp., image transformations) 
$q$ and continuous k-proper functions $w: Y \longrightarrow X^{\flat}$ (resp.,   $w: Y \longrightarrow X^\sharp$).
Moreover,
\begin{enumerate}[leftmargin=0.25in, label=(h\arabic*),ref=(h\arabic*)]
\item \label{ITqh1}
$\theta(f)(y) =  \int_X f \, dw_y = \int_X f \, d(q^*\delta_y)$. 
\item \label{ITqh1.1}
For a quasi-homomorphism $\theta$, $\theta(f)(y) =  t_0 \in \overline{f(X)}$, where 
$ t_0 = \inf \{ t : w_y(\{ x: f(x) <t\}) =1 \} =  \sup \{ t : w_y(\{ x: f(x) <t\}) =0 \} = \inf \{ t : w_y(\{ x: f(x) > t\}) =0 \} = \sup \{ t : w_y(\{ x: f(x) >t\}) =1 \}$.
\item \label{ITqh1a}
$w = q^* \circ i_Y$, $ q = w^{-1} \circ \Lambda$.
\item \label{qhMon}
$ f \le g$ implies $ \theta(f) \le \theta(g)$.
\item \label{ITqh6}
$ \| \theta(f) \| \le \| f\|$. 
\item \label{ITqh1b} 
If $f_{\alpha} \nearrow f $ then $ \theta(f_{\alpha}) \nearrow \theta(f)$. 
For a quasi-homomorphism, if  $f_{\alpha} \searrow f $ then $ \theta(f_{\alpha}) \searrow \theta(f)$. 
\item \label{ITqh6a}
If $ f, g \in C_c^+(X)$ then $ \| \theta(f) - \theta(g) \| \le \| f -g\|$, and if $X$ is compact, then 
$ \| \theta(f) - \theta(g) \| \le \| f -g\|$ for any $f,g \in C(X)$. 
For a quasi-homomorphism, $\| \theta(f) - \theta(g) \| \le 2 \| f -g\|$. 
\item \label{ITqh2}
For an open set $U \subseteq X$, $q(U)$ is an open $F_{\sigma}$ set
\begin{align*} 
q(U) &= w^{-1} (U^\sharp) (\mbox{resp., }  w^{-1} (U^\flat))  = \bigcup\{ \theta(f)^{-1} (\{1\}): 0 \le f \le 1_U,  f \in C_c(X) \} \\
&= \bigcup\{ \theta(f)^{-1} ((0, \infty)): 0 \le f \le 1_U,  f \in C_c(X) \}.
\end{align*}
\item \label{ITqh3}
For a compact $K \subseteq X$
\begin{align*} 
q(K) &=  w^{-1} (K^\sharp)  (\mbox{resp., }  w^{-1} (K^\flat))  =\bigcap \{ \theta(f)^{-1} (\{ 1\}): 1_K \le f \le 1,   f \in C_c(X) \}    \\
& = \bigcap \{ \theta(f)^{-1} ([1, \infty)): 1_K \le f \le 1,   f \in C_c(X) \}. 
\end{align*}
\item \label{ITqh4}
For an image transformation $q$ and an open set $ A \subseteq \mathbb{R}$ or a closed set $A \subseteq \mathbb{R} \setminus \{ 0\}$ 
\begin{align} \label{qthetaf}
q(f^{-1}(A)) = (\theta(f))^{-1} (A),
\end{align}
and for compact $X$ this holds for any open or closed $ A \subseteq \mathbb{R}$. For a d-image transformation $q$ equality (\ref{qthetaf}) holds for
$A=[t, \infty)$ and $A=(t, \infty)$ for  almost every $t$ .
\item \label{ITqh4a}
For a d-image transformation $q$, $q(f^{-1}(A)) \subseteq (\theta(f))^{-1} (A)$ for any open $ A \subseteq \mathbb{R}$ 
and any closed set $A \subseteq \mathbb{R} \setminus \{ 0\}$.
\item \label{ITqh12}
$q(X) = Y$.
\item \label{ITqh15}
$(q(E))^{\flat}  = (q^*)^{-1}(E^{\flat})$ (resp., $(q(E))^{\sharp}  = (q^*)^{-1}(E^{\sharp})$ for a compact or an open set $E \subseteq X$.
\item \label{ITqh5}
$ \int_X f \, dq^*\nu = \int_Y \theta(f) \, d\nu \  $.
\item \label{ITqh7}
$\theta^* = q^*$.
\end{enumerate}
\end{theorem}

\begin{proof}
The proof for image transformations and d-image transformations is similar, and we shall prove the theorem for image transformations.
We again may write $w_y$ for the simple topological measure $w(y)$, where $w: Y \longrightarrow X^\sharp$ as in 
Theorems \ref{ITw}  and  \ref{wEq}.  \\
The stated 1-1 correspondence, \ref{ITqh1}, and \ref{ITqh1a} - \ref{ITqh6a}
are from Theorems \ref{ITw} and \ref{wEq}. \\
\ref{ITqh1.1}. Follows from \cite[Th. 6.9, L. 6.11]{Butler:ReprDTM}. \\
\ref{ITqh2}. For an open or a compact set $A \subseteq X$  by Theorem \ref{ITw} 
 $q(A) =  w^{-1}(A^\sharp)$.
Let $U \subseteq X$ be open. If $y \in q(U)$ then  $w_y(U) = 1$, 
so there is a compact $ C  \subseteq U$ such that $w_y(C) = 1$. 
For a Urysohn function $f$ with $1_C \le f \le 1_U$ 
by Remark \ref{RemBRT}\ref{mrDTM}   we have
$ 1 = w_y (C) \le \int_X f \, dw_y  \le w(U) = 1,$
so $\theta(f)(y) =1$, and $y  \in \theta(f)^{-1} (\{1\}) \subseteq \theta(f)^{-1} ((0, \infty)) $. 

If $0 \le f \le 1_U, f \in C_c(X)$ then for each $y \in Y$ by Remark \ref{RemBRT}\ref{mrDTM}  
$ \theta(f)(y) = \int_X  f \, dw_y \le w_y(U).$
If $ \theta(f)(y) >0$ then $w_y(U) >0$, i.e. $w_y (U) = 1$, i.e. $ w(y)\in U^\sharp$, i.e. $y \in w^{-1}(U^\sharp) = q(U)$. 
Thus, $(\theta(f))^{-1} ((0, \infty)) \subseteq q(U)$.  \\
\ref{ITqh3}. 
Let $K \subseteq X$ be compact.  
If $y \in q(K) = w^{-1}(K^\sharp)$, i.e. $w_y(K) =1$, then for  any $ f \in C_c(X)$ such that $ 1_K \le f \le 1$ we have 
$ 1= w_y(K) \le \int_X f\, dw_y  = \theta(f)(y) \le 1.$
Thus,  $y \in   \theta(f)^{-1}(\{1\}) \subseteq \theta(f)^{-1}([1, \infty))$. 
If $ y \in \bigcap \{ \theta(f)^{-1} ([1, \infty)): 1_K \le f \le 1,   f \in C_c(X) \}$ then $\int_X f\, dw_y \ge 1$ for each $  f \in C_c(X)$ such that $1_K \le f \le 1$.  
By Remark \ref{RemBRT}\ref{mrDTM}  $\, w_y(K)\ge 1$, so 
$w_y(K) = 1$, and $y \in w^{-1}(K^\sharp) = q(K)$. \\
\ref{ITqh4}. 
By part \ref{ITqh1}, Remark \ref{RemBRT}\ref{prt1}, and Theorem \ref{rhosimple}\ref{ptm} 
the point mass on $\mathbb{R}$ corresponding to $w_y$ is $ m_{w_y, f} =  \delta_{\theta(f)(y)} = \delta_{\int f \, dw_y}$,
and on specified subsets of $ \mathbb{R}$ by \cite[Th. 6.10]{Butler:ReprDTM}   $m_{w_y, f} = w_y \circ f^{-1} = w(y) \circ f^{-1}$. 
Note that for a closed set $A \subseteq \mathbb{R} \setminus \{ 0\}$, $ f^{-1}(A)$ is compact. 
Then 
 $ y \in (\theta(f))^{-1}(A)  \Longleftrightarrow  \delta_{\theta(f)(y)} (A) = 1  \Longleftrightarrow  (w(y) \circ f^{-1})(A) = 1   
\Longleftrightarrow y \in w^{-1} ((f^{-1}(A))^\sharp) = q(f^{-1}(A)) $ by parts \ref{ITqh2}, \ref{ITqh3}.
The proof for d-image transformations is similar, but uses \cite[Th. 6.7]{Butler:ReprDTM} instead of \cite[Th. 6.10]{Butler:ReprDTM}.  \\ 
\ref{ITqh4a}. As in proof of part \ref{ITqh4}, $ r_{w_y, f} =  \delta_{\theta(f)(y)}$. By \cite[Th. 6.7]{Butler:ReprDTM}  
$w_y(f^{-1}(A) \le r_{w_y, f}(A)$. Using \ref{ITqh2} and \ref{ITqh3},
$ q(f^{-1}(A)) = \{ y: w_y((f^{-1}(A))) = 1 \} \subseteq  \{ y: r_{w_y, f}(A) = 1 \} = \{ y: \theta(f)(y)  \in A \} = \theta(f) ^{-1} (A)$. \\
\ref{ITqh12}. Follows from \ref{ITqh2} or \ref{ITqh4}.\\
\ref{ITqh15}.
$\nu \in (q(E))^{\sharp} \iff \nu(q(E)) =1 \iff q^* \nu \in E^{\sharp} \iff \nu \in (q^*)^{-1} (E^{\sharp})$, so
$(q(E))^{\sharp} =(q^*)^{-1} (E^{\sharp})$. \\
\ref{ITqh5}.  
By Remark \ref{RemBRT}\ref{prt1} and part \ref{ITqh4} 
the distribution functions  $R_{1, q^*\nu, f}(t) $ and $R_{1, \nu, \theta(f)} (t)$  are equal 
(in case of deficient topological measures, they are equal for almost all $t$). By part  \ref{ITqh12} $(q^*\nu)(X) = \nu(Y)$. 
In formula (\ref{rfform}) we can use the same interval of integration for both integrals, so the integrals are equal.  \\
\ref{ITqh7}. Both $q^*$ and $\theta^*$ map $\mathbf{TM}(Y)$ to   $\mathbf{TM}(X)$ (or  $\mathbf{DTM}(Y)$ to   $\mathbf{DTM}(X)$). 
Let $ \nu$ be a (deficient) topological measure on $Y$. 
By (\ref{Intdq*})  and part \ref{ITqh5} $\int_X f\, dq^* \nu = \int_X f\, d\theta^* \nu$ for any $ f \in C_0^+(X)$.
By (\ref{EqByInt})  $q^* \nu = \theta^* \nu$. Thus,   $q^* = \theta^*$. 
\end{proof}

\begin{remark}
From part \ref{ITqh5} we see that $\nu_{\alpha} \Longrightarrow \nu$ implies $q^*(\nu_{\alpha}) \Longrightarrow q^*(\nu)$, so $q^*$ is continuous,
which we also proved in Theorem \ref{adjCont}.  
\end{remark}

\begin{remark} \label{conicqlfExt}
A conic quasi-homomorphism $\theta : C_0^+(X) \longrightarrow C_0^+(Y)$ corresponding to a d-image transformation $q$ 
can be extended to a conic quasi-homomorphism
$\theta : C_0(X) \longrightarrow C_b(Y)$  by letting 
$\theta(f)(y) =  \int_X f \, dw_y = \int_X f \, d(q^*\delta_y)$ for $f \in C_0(X)$ (see Theorem \ref{wEq} and Theorem \ref{ITw}).  
Then parts  \ref{qhMon}, \ref{ITqh1b}, \ref{ITqh4}, \ref{ITqh4a}, and \ref{ITqh5}  of Theorem \ref{ITqh} hold for function from $C_0(X)$
(see Remark \ref{RemBRT}\ref{RHOsvva}, Theorem \ref{INTsmoo}, and proof of Theorem \ref{ITqh}).
\end{remark}
 
\begin{example}
Let $\mu_0$ be a simple topological measure on a LC space $X$, let $Y$ be compact, and let $w(y) = \mu_0$ for all $y \in Y$, 
i.e. $q^* \delta_y = \mu_0$ for all $y$. 
By Theorem \ref{ITqh}\ref{ITqh1} $ \theta(f) \equiv \int f \, d\mu_0$, a constant function. 
\end{example}  

\begin{example} \label{qhOnXY}
Let $X,Y$ be LC, and let $ \mu$ be a  finite topological measure on $X$.  
For $ f \in C_c(X \times Y)$ let $\theta(f)(y) = \int_X f_y \, d\mu$, where $f_y \in C_c(X)$ is given by $f_y(x) = f(x,y)$. 
By \cite[Pr. 4.1]{Butler:RepeatedQint} $\theta(f) \in C_c(Y)$. 
It is easy to check that $\theta: C_c(X \times Y) \rightarrow C_c(Y)$ is a bounded quasi-linear map. 
By Theorem \ref{extfromCcOp} we can extend $\theta$ to a quasi-linear map on $C_0(X \times Y)$. 
If $\eta$ is a quasi-linear functional on $Y$ then $ \eta \circ \theta$ is a quasi-integral 
iff $\eta$ is linear or $ \mu$ is a nonnegative scalar multiple of a simple topological measure, see \cite[Th. 4.1]{Butler:RepeatedQint}. 
Now suppose $ \mu$ is a simple topological measure.
Applying Theorem \ref{rhosimple} to $ \int_X f_y \, d\mu$, we see that $ \theta$ is a quasi-homomorphism.   
For a corresponding  image transformation $q$ and an open $U \subseteq X \times Y$ by Theorem \ref{ITqh}\ref{ITqh2} and \cite[Th. 49]{Butler:Integration}
with  $U_y =\{x: (x,y) \in U\}$ we have:
\begin{align*}
&q(U) = \bigcup_{f \in C_c^+(X \times Y), f \le 1_U} \{ y \in Y: \int_X f_y \, d\mu > 0 \}  \\
& =  \bigcup_{f \in C_c^+(X \times Y),  f \le 1_U} \{ y \in Y: \mu( Coz (f_y))  = 1\} =  \{ y \in Y: \mu(U_y) = 1\}.
\end{align*}
\end{example} 

\begin{definition} \label{DefMarkovOp}
Let $X,Y$ be locally compact. 
An operator $S: DTM(Y) \longrightarrow DTM(X)$ is called a Markov operator if 
$ S(a \mu + b \nu) = a S(\mu) + b S(\nu)$ for $a,b \ge 0, \mu, \nu \in DTM(Y)$.
A Markov operator $S: DTM(Y) \longrightarrow DTM(X)$ is called Markov-Feller operator if there is an operator
$T: C_0(X) \longrightarrow C_0(Y)$ such that $ \int_X f \, d(S(\mu)) = \int_Y T(f) \, d\mu $ for $g \in C_0(X), \mu \in DTM(Y)$.  
The operator $T$ is called dual to $S$.
\end{definition}

\noindent
Definition \ref{DefMarkovOp} is closely related to \cite[p. 345]{LasotaMyjakSzarek}. 

\begin{example} \label{q*MarFel}
If $X = Y$ in Theorem  \ref{ITqh}, we see that the operator on finite (deficient) topological measures $q^*$ is a Markov-Feller operator, 
whose (nonlinear!) dual operator is $\theta$; see  Theorem  \ref{ITqh} \ref{ITqh5}, Remark \ref{conicqlfExt}, and \cite[p. 345]{LasotaMyjakSzarek}. 
\end{example}

\begin{lemma} \label{ITcompQH}
Suppose $X,Y, Z$ are LC, $q_1$ is a (d-) image transformation from $X$ to $Y$, 
and $q_2$ is a (d-) image transformation from $Y$ to $Z$, with 
corresponding (conic) quasi-homomorphisms $\theta_1$ and $\theta_2$. Then for the  (d-) image transformation $ q = q_2 \circ q_1$ the corresponding 
(conic) quasi-homomorphism $\theta$ has the form $ \theta =  \theta_2 \circ \theta_1$ (respectively,  $ \theta =  \theta_2 \circ \theta_1$  on $C_0^+(X)$ 
for conic quasi-homomorphisms).  
\end{lemma}
 
\begin{proof}
For   $ f \in C_0(X)$ (resp.,   $ f \in C_0^+(X)$) and $z \in Z$  using Theorem  \ref{ITqh}\ref{ITqh1},\ref{ITqh5} we have:
\begin{align*}
\theta(f)(z) &= \int_X f\, d((q_2 \circ q_1)^* \delta_z)  = \int_X f\, d(q_1^*( q_2^* \delta_z)) 
 =  \int_Y \theta_1(f) \, d(q_2^* \delta_z)) \\
 &= \theta_2 ( \theta_1(f))(z). 
\end{align*} 
Thus, $\theta(f) = (\theta_2 \circ \theta_1) (f)$ for all $ f \in C_0(X)$ (resp.,  for all $ f \in C_0^+(X)$). 
\end{proof}

\begin{remark}
Using Lemma \ref{ITcompQH} we see (as indicated in \cite{Taraldsen:RegQHinLC}) that the map $q \longmapsto \theta$ gives a functor from 
the category of image transformations as objects and compositions as morphisms to 
the category of quasi-homomorphisms as objects and compositions as morphisms.  
\end{remark} 
   
\begin{theorem} \label{q11eq}
Let $X$, $Y$ be LC.  Suppose $q$ is an image transformation from $X$ to $Y$ with the corresponding quasi-homomorphism $\theta: C_0(X) \longrightarrow C_0(Y)$
or  $q$ is a d-image transformation and  $\theta: C_0(X) \longrightarrow C_b(Y)$  is from Remark \ref{conicqlfExt}. TFAE:
\begin{enumerate}[leftmargin=0.25in, label=(\roman*),ref=(\roman*)]
\item \label{qxnon0}
$q(\{x\}) \neq \emptyset$ for any $ x \in X$.
\item \label{q1to1}
$q$ is 1-1.
\item \label{thetaISO}
$\| f - g \| \le \| \theta(f) - \theta(g) \|$ (hence, $\theta$ is an isometry on $C_c^+(X)$; on $C(X)$ if $X$ is compact). 
\item \label{theta11}
$ \theta$ is 1-1.
\end{enumerate}
\end{theorem}

\begin{proof}
The equivalence of \ref{qxnon0} and \ref{q1to1} is a part of Lemma \ref{ITsvva}. 
\ref{qxnon0}  $\Longrightarrow$  \ref{thetaISO}. For $f, g \in C_0(X)$ let $x$ be such that $ \| f - g \| = | f(x) - g(x)|$. Let $y \in q(\{x\})$. 
By Theorem \ref{adjCont} $\delta_x = q^*(\delta_y)$, so by Theorem \ref{ITqh}\ref{ITqh5} and Remark \ref{conicqlfExt} we have: 
\begin{align*}
\| f - g \| &= | f(x) - g(x)| = | \int_X f \, d q^*(\delta_y) - \int_X g \, d q^*(\delta_y) | \\
&= | \int_Y \theta(f) \, d \delta_y - \int_Y \theta(g) \, d \delta_y |  = | \theta(f)(y) - \theta(g)(y) | \le \| \theta(f) - \theta(g) \|. 
\end{align*}
The statement about isometry follows from Theorem \ref{wEq}\ref{wEqII}. 
\ref{thetaISO}   $\Longrightarrow$ \ref{theta11} is clear.
\ref{theta11} $ \Longrightarrow$ \ref{qxnon0}. 
Suppose $q(\{x\} = \emptyset$.  Let $V$ be a neighborhood of $x$ with compact closure.
Using Lemma \ref{easyLeLC} and Lemma \ref{ITsvva} write 
$q(\{x\}) =
\bigcap \{ q(\overline{U}): x \in U  \subseteq \overline{U}  \subseteq V \} $. By compactness of $ q(\overline{V})$ there are 
neighborhoods $\  U_1, \ldots, U_n$  of $x$ such that $ q(\overline U_1)\cap \ldots \cap q( \overline U_n) = \emptyset$. Then  $q(U) = \emptyset$ 
for nonempty $U =U_1 \cap \ldots \cap U_n$. 
Choose $f \in C_c^+(X)$ with $f(x) = 1, \,  supp f \subseteq U$. Since $ \theta (f) \ge 0$ and 
by  assumption $ \theta(f)$ is nontrivial, 
$(\theta(f)^{-1}(\{a\}) \neq \emptyset$ for some $ a >0$. But by Theorem \ref{ITqh}\ref{ITqh4} and Remark \ref{conicqlfExt} 
$(\theta(f)^{-1}(\{a\}) = q(f^{-1} (\{a\}))\subseteq q(U) = \emptyset$, a contradiction. 
\end{proof}
   
\begin{theorem} \label{9eqThet}
Let $X$, $Y$ be LC.  Suppose $q$ is an image transformation from $X$ to $Y$ with the corresponding function $w$ and quasi-homomorphism $\theta: C_0(X) \longrightarrow C_0(Y)$ from Theorem \ref{ITqh}
or  $q$ is a d-image transformation with the corresponding $w$ and conic  quasi-homomorphism 
$\theta: C_0(X) \longrightarrow C_b(Y)$  from Remark \ref{conicqlfExt}. TFAE:
\begin{enumerate}[label=(\roman*),ref=(\roman*)]
\item \label{thet0}
Each $w_y$ is a measure.  
\item \label{thet1}
$\theta$ is linear.
\item \label{thet2}
$\theta$ is an algebra homomorphism.
\item \label{thet3}
$\theta^*(P_e(Y)) = q^*(P_e(Y))  \subseteq P_e(X)$.
\item \label{thet4}
$Y = \bigcup_{x \in X} q(\{x \})$.
\item \label{thet5}
There is a continuous proper function $u:Y \rightarrow X$ such that $ q = u^{-1}$.
\item \label{thet6}
$q(A \cup B) = q(A) \cup q(B)$ whenever $A, B, A \cup B \in \mathscr{O}(X) \cup \mathscr{K}(X)$.
\item \label{thet7}
$q(C \cup K) = q(C) \cup q(K)$ whenever $C,K \in \mathscr{K}(X)$.
\item \label{thet8}
$q(U \cup V) \subseteq q(U) \cup q(V)$, hence, $q(U \cup V) = q(U) \cup q(V)$  whenever $U, V \in \mathscr{O}(X) $.
\end{enumerate}
\end{theorem}

\begin{proof}
\ref{thet0} $\Longrightarrow$ \ref{thet1} is clear because $\theta(f)(y) = \int_X f\, dw_y$.  
\ref{thet1} $\Longrightarrow$ \ref{thet2} follows from Remark \ref{ThetMul}.
\ref{thet2} $\Longrightarrow$ \ref{thet3}. 
For $y \in Y$ the functional $ \rho$ given by $ \rho(f) = \theta(f)(y)$ is linear and multiplicative, so $ \rho(f) = \int_X f\, d \delta_x$ for some point mass
$\delta_x$. By Theorem \ref{ITqh}\ref{ITqh1} $\rho(f) =  \int_X f \, d(q^*\delta_y) $ for all $ f \in C_0^+(X)$. 
By (\ref{EqByInt}) and Theorem \ref{ITqh}\ref{ITqh7} $\theta^* \delta_y =q^*\delta_y = \delta_x $.    
\ref{thet3} $\Longleftrightarrow$ \ref{thet4} $\Longleftrightarrow$  \ref{thet5}. See Theorem \ref{ITfromUf}. 
\ref{thet5} $\Longrightarrow$ \ref{thet6}  $\Longrightarrow$ \ref{thet7} is clear.
\ref{thet7} $\Longrightarrow$ \ref{thet8} follows from Definition \ref{IT} since a compact $K \subseteq U \cup V$ can be written as
$K = C \cup D, C  \subseteq U, D \subseteq V, C,D  \in \mathscr{K}(X)$.  
\ref{thet8} $\Longrightarrow$ \ref{thet0}.  Let $y \in Y$, and $U,V \in \mathscr{O}(X)$. For the (deficient) topological measure $w_y$ we have
$w_y(U \cup V) = (q^* \circ \delta_y) (U \cup V) = \delta_y (q(U) \cup q(V)) \le w_y(U) + w_y(V)$, 
and by Theorem \ref{subaddit} $w_y$ is a measure. 
\end{proof} 

\begin{corollary} \label{ontoBij}
Suppose $X$, $Y$ are LC, (a)  a quasi-homomorphism $\theta: C_0(X) \longrightarrow C_0(Y)$ is onto or 
(b) a conic quasi-homomorphism  $\theta': C_0^+(X) \longrightarrow C_0^+(Y)$ is onto, 
and $ \theta: C_0(X) \longrightarrow C_b(Y)$ is an extension of $\theta'$ from Remark \ref{conicqlfExt}.
Then $\theta: C_0(X) \longrightarrow C_0(Y)$ is an algebra isomorphism. 
\end{corollary} 

\begin{proof}
First we shall show that $q^*(P_e(Y) \subseteq P_e(X)$. 
$\theta$ (resp., $ \theta'$) is onto, and  by Theorem \ref{qhEqDef}
there is an inverse map
$\xi : C_0(Y) \longrightarrow C_0(X)$ which is a quasi-homomorphism (resp. $\xi : C_0^+(Y) \longrightarrow C_0^+(X)$, a conic quasi-homomorphism).
For a (deficient) topological measure $\mu$ on $X$ for each $f \in C_0^+(X)$  by Theorem \ref{ITqh}\ref{ITqh5},\ref{ITqh7}
\begin{align*}
\int_X f \, d\mu = \int_X \xi( \theta(f)) \, d\mu = \int_Y \theta(f) \, d(\xi^* \mu) = \int_X f\, d(\theta^* \xi^* \mu),
\end{align*} 
so by (\ref{EqByInt}) $ \mu  = \theta^* \xi^* \mu$. Thus, $\theta^* \xi^* = id$. Similarly, $ \xi^*  \theta^*= id$.

Let $ y \in Y$, and let $ \sigma = q^* \delta_y$. For the (d-) image transformations $p$ and $q$ corresponding to $ \xi$ and $\theta$  respectively, 
by Theorem \ref{ITqh}\ref{ITqh7} 
$ \  \sigma( p( \{y\})) = p^*\sigma(\{y\}) = (p^* q^* \delta_y)(\{y\}) = (\xi^* \theta^* \delta_y)  (\{y\}) = \delta_y (\{y\}) = 1,$
so $ p( \{y\}) \neq \emptyset$. For $x \in  p( \{y\})$ by Theorem \ref{adjCont} we have $\delta_y = p^* \delta_x$, and then 
$ q^* \delta_y =  q^*  p^* \delta_x = \theta^* \xi^*  \delta_x = \delta_x \in P_e(X)$.

Since $q^*(P_e(Y) \subseteq P_e(X)$, by Theorem \ref{9eqThet} $\theta$ is linear, so $\theta(f) = \theta'(f^+) - \theta'(f^-) \in C_0(Y)$. 
Then $\theta: C_0(X) \longrightarrow C_0(Y)$ is an algebra isomorphism by Theorem \ref{9eqThet}.
\end{proof}

\begin{remark} \label{PsiLam}
Let $ q$ be the (d-) image transformation corresponding to the  (conic) quasi-homomorphism $\Psi$ from Proposition \ref{Psifhat}. 
For each compact $K$ by Theorem \ref{ITqh}\ref{ITqh3} and Remark \ref{RemBRT}\ref{mrDTM}
we see that  $q(K) = \{ \mu: \tilde f(\mu) = 1 \mbox{   for each   } f \in C_c(X), 1_K \le f \le 1\} = K^\sharp$ (respectively, $K^{\flat}$).
Thus, $q = \Lambda$, the basic (d-) image transformation. We call $\Psi$ the basic (conic) quasi-homomorphism.
\end{remark}

\begin{theorem}  \label{LambCont}
Suppose $X$ is LC,   $\Lambda$ is the basic (d-) image transformation, and 
$\Psi: C_0(X) \longrightarrow C(X^{\sharp})$  ($\Psi: C_0(X) \longrightarrow C(X^{\flat})$)  is the basic (conic) quasi-homomorphism.
Consider the following statements.
\begin{enumerate}[leftmargin=0.25in, label=(\Roman*),ref=(\Roman*)]
\item \label{laba1}
 $\Lambda$ is an inverse of a continuous proper function $u: X^\sharp \Longrightarrow X$ (resp., $u: X^\flat \Longrightarrow X$).
\item \label{laba2}
 $P_e(X) = X^\sharp$ (respectively,  $P_e(X) = X^{\flat}$).
\item \label{laba3} 
Each representative (deficient) topological measure on $X$ is a Borel measure.
\item \label{laba4}
Each simple  (deficient) topological measure  is subadditive on $\mathscr{O}(X) \cup \mathscr{K}(X)$.
\item \label{laba5}
Each simple  (deficient) topological measure  is subadditive on $\mathscr{O}(X)$.
\item \label{laba6}
Each simple  (deficient) topological measure  is subadditive on $\mathscr{K}(X)$.
\item \label{laba7}
$(A \cup B)^\sharp = A^\sharp \cup B^\sharp$ (resp., $(A \cup B)^{\flat} = A^{\flat} \cup B^{\flat}$) for $A, B, A \cup B \in \mathscr{O}(X) \cup \mathscr{K}(X)$.
\item \label{laba8}
$(K \cup C)^\sharp = K^\sharp \cup C^\sharp$ (resp., $(K \cup C)^{\flat} = K^{\flat} \cup C^{\flat}$)  for $K,C \in \mathscr{K}(X)$. 
\item \label{laba9}
$(U \cup V)^\sharp = U^\sharp \cup V^\sharp$  (resp., $ (U \cup V)^{\flat} = U^{\flat} \cup V^{\flat}$)  for $U, V \in \mathscr{O}(X)$. 
\item \label{Ps4u}
Each  representable quasi-integral on $X$ is linear.
\item \label{Ps5u}
$\Psi$ is linear.
\item \label{Ps2u} 
$\Psi$ is an algebra  isomorphism. 
\item \label{Ps1u}
$ \Psi$ is onto.
\item \label{PsOponto}
If $W$ is open in  $X^{\flat}$ (resp., in $X^\sharp$ ) then $W = U^{\flat} $  (resp., $ W = U^\sharp$)  for some $U \in \mathscr{O}(X)$. 
\item \label{Lambonto}
$\Lambda$ is onto.
\end{enumerate} 
Statements \ref{laba1} - \ref{Ps2u} are equivalent and imply  \ref{PsOponto}. For the basic quasi-homomorphism $ \Psi$ statements  
\ref{laba1} - \ref{Ps1u}  are equivalent and imply  \ref{PsOponto}, 
and if $X$ is compact then 
\ref{laba1} -\ref{Lambonto}  are equivalent. 
\end{theorem}  

\begin{proof}
The proof of equivalence of \ref{laba1} - \ref{Ps2u}  is similar for topological measures and deficient topological measures.
 \ref{laba1}  $\Longleftrightarrow$ \ref{laba2} is in Theorem \ref{ITfromUf}.
\ref{laba2} $\Longrightarrow$ \ref{laba3}  $\Longrightarrow$ \ref{laba4}  $\Longrightarrow$ \ref{laba5} is clear, and
\ref{laba5}  $\Longrightarrow$ \ref{laba6} $\Longrightarrow$ \ref{laba3} is by Theorem \ref{subaddit}.
\ref{laba4}  $\Longrightarrow$ \ref{laba7}: If $ \mu \in (A \cup B)^\sharp$ then $1 = \mu(A \cup B) \le \mu(A) + \mu(B)$, so
at least one of  
$\mu(A), \mu(B)$ must be $1$, so $ \mu \in A^\sharp \cup B^\sharp \subseteq (A \cup B)^\sharp$.
\ref{laba7} $ \Longrightarrow$ \ref{laba8} $ \Longrightarrow$ \ref{laba9}  $ \Longrightarrow$ \ref{laba1} is given by Theorem \ref{9eqThet}. 
\ref{laba3} $\Longrightarrow$ \ref{Ps4u}:   Clear, since in this case each representable quasi-integral is an integral with respect to a measure. 
\ref{Ps4u} $\Longrightarrow$ \ref{Ps5u}  follows from $ \Psi(f+g)( \mu) = \Psi(f) (\mu) + \Psi(g)(\mu)$ for each simple $\mu$. 
\ref{Ps5u}  $\Longrightarrow$ \ref{Ps2u}  $\Longrightarrow$ \ref{laba1} by Theorem \ref{9eqThet} and Proposition \ref{Psifhat}. 
\ref{laba2}  $\Longrightarrow$  \ref{PsOponto}: By  \ref{laba2} each simple deficient topological measure on $X$ is a point mass, and 
a basic open set $W$ in  $X^{\flat}$ from Remark \ref{basicNbd*} has the form
$W = \{ \delta_x: \delta_x \in U_i^{\flat}, i =1, \ldots, n, \delta_x \notin C_j^{\flat}, j =1, \ldots, m. \} $, 
i.e.  $W = U^{\flat}$ for an open set $U = (U_1 \cap \ldots \cap U_n) \setminus (C_1 \cup \ldots \cup C_m)$. 
For an arbitrary open $W = \bigcup W_s = \bigcup U_s^{\flat}$, a simple deficient topological measure
$ \delta_x \in W \Longleftrightarrow x \in U_s$ for some $s$  $ \Longleftrightarrow x \in U = \bigcup U_s$. 

Now let $ \Psi$ be the basic quasi-homomorphism. 
\ref{Ps5u} $\Longrightarrow$ \ref{Ps1u}:  
By Proposition \ref{Psifhat}, $ \mathbf{B} = \Psi(C_0(X))$ is closed in $C(X^\sharp)$. 
By Theorem \ref{rhosimple} $ \mathbf{B} $  contains squares, so by Remark \ref{ThetMul}, $\mathbf{B}$ is an algebra. 
Also,  $ \mathbf{B} $ separates points in $X^\sharp$ and vanishes identically at no $ \mu  \in X^\sharp$. 
Then the Stone-Weierstrass theorem (\cite[(7.37)]{HS})  gives \ref{Ps1u}.
\ref{Ps1u} $\Longrightarrow$ \ref{Ps2u} by Corollary \ref{ontoBij}.
If $X$ is compact, then $X^\sharp = U^\sharp \sqcup (X \setminus U)^\sharp$ for any open $U$, and  \ref{PsOponto}  $\Longrightarrow$ \ref{Lambonto};
\ref{Lambonto} $\Longrightarrow$ \ref{laba2} by \cite[Pr. 5.1]{Aarnes:Pure}.
\end{proof}

\begin{remark}
On $\mathbb{R}^2$ or a square in $\mathbb{R}^2$ there are simple (deficient) 
topological measures that are not point masses; on $\mathbb{R}^1$ there are simple deficient topological measures that are not point masses 
(see, for instance, Examples \ref{ExDan2pt} -\ref{basicExDTM}). 
By Theorem \ref{LambCont} on these spaces the basic (d-) image transformation $\Lambda$ is not an inverse of a continuous proper function. 
If $X$ is a closed interval on the line, then every topological measure is a measure (see \cite[Th. 3]{Grubb:SignedqmDimTheory}), so 
$P_e(X) = X^\sharp$, and the basic image transformation $\Lambda$ is an inverse of a continuous proper function. 
\end{remark}

\begin{example} \label{exqxempt}
Let $Y $ be a closed subset of $X^\sharp$ (respectively, of $X^\flat$.) 
Taking $w :Y \longrightarrow X^\sharp$ (resp., $w :Y \longrightarrow X^\flat$)  to be the inclusion map, we obtain 
a (d-) image transformation $q = w^{-1} \circ \Lambda$, and $q(A) = A^\sharp \cap Y$ (resp., $ A^\flat \cap Y$).
If $X$ has simple (deficient) topological measures that are not
point masses (see, for instance, \cite[Ex. 6.1]{Butler:DTMLC}) and $ Y \cap P_e(X) = \emptyset$ then $q(\{x\}) = \emptyset$ for all $ x \in X$. 
\end{example} 

\begin{theorem} \label{uThet11}
Suppose $X$, $Y$ are  LC spaces. There is a 1-1 correspondence between continuous proper maps $u: Y \rightarrow X$ and 
algebra homomorphisms $\theta: C_0(X) \rightarrow C_0(Y)$, and $\theta(f) = f \circ u$. 
For $ f \in C_0(X)$ and a topological measure $\nu$ on $Y$ we have
$$ \int_Y (f \circ u) \, d\nu = \int_X f\, d(\nu \circ u^{-1}).$$ 
\end{theorem}

\begin{proof}
Given a continuous proper  $u: Y \rightarrow X$ define  $\theta(f) = f \circ u$ for $f \in C_0(X)$. 
Since $u$ is proper, $\theta(f) \in C_0(Y)$. Clearly,  $\theta: C_0(X) \rightarrow C_0(Y)$  is an algebra  homomorphism.
Now let $\theta: C_0(X) \rightarrow C_0(Y)$  be an algebra  homomorphism. By Theorem \ref{9eqThet} the corresponding image transformation
$q= u^{-1}$ for some continuous proper function  $u: Y \rightarrow X$. For each $y$ by Theorem \ref{ITqh}\ref{ITqh1} 
$$\theta(f)(y) = \int_X f \, d(q^* \delta_y) = \int_X f \, d(\delta_y \circ u^{-1})  = \int_Y (f \circ u) \, d \delta_y = (f \circ u) (y),$$
hence,  $\theta(f) = f \circ u$. 
The last statement follows from Theorem \ref{ITqh}\ref{ITqh5} and Remark \ref{invfunDTM}.
\end{proof} 
 
The next theorem gives the structure of a bounded quasi-linear  map or a conic quasi-linear map $\theta$  as 
$ \theta = H \circ \Psi$, where $ H$ is an algebra homomorphism.
 
\begin{theorem} \label{QH1to1}
Let $X, Y$ be LC. 
\begin{enumerate}[leftmargin=0.2in, label=(\Roman*),ref=(\Roman*)]
 \item \label{QH1to1a}
 Suppose  $\theta: C_0(X) \longrightarrow C_b(Y)$ is a bounded quasi-linear map (respectively, 
a conic quasi-linear map).
Then there is an algebra homomorphism
$H : C(\mathcal M) \longrightarrow C_b(Y)$ such that $ \theta = H \circ \Psi$ (resp., $ \theta = H \circ \Psi$ on $C_0^+(X)$),
where $ \mathcal M $ is a uniformly bounded in variation family of  topological (resp., deficient topological) measures on $X$.  
\item \label{QH1to1b}
If  $\theta: C_0(X) \longrightarrow C_0(Y)$ is a quasi-homomorphism (resp.,  
$\theta: C_0(X) \longrightarrow C_b(Y)$ is 
a conic quasi-homomorphism from Remark \ref{conicqlfExt})
then  $H$ is defined on $X^\sharp$ (resp., on $X^\flat$) and
$H(\tilde{f}) = \tilde{f} \circ \theta^* \circ i_Y$ for each $ f \in C_0(X)$. 
\item \label{QH1to1c}
If  $H : C(X^\sharp) \longrightarrow C_0(Y)$ (resp.,  $H : C(X^\flat) \longrightarrow C_0(Y)$)  is an algebra homomorphism,
then there is a  quasi-homomorphism (resp., a conic quasi-homomorphism)  
$\theta: C_0(X) \longrightarrow C_0(Y)$ such that $H(g) = g \circ \theta^* \circ i_Y$. 
\end{enumerate}
In particular, if $Y$ is compact, then there is a 1-1 correspondence between 
quasi-homomorphisms $\theta: C_0(X) \longrightarrow C(Y)$ and algebra homomorphisms 
$H : C(X^\sharp) \longrightarrow C(Y)$; moreover, $ \theta = H \circ \Psi$ and $H(\tilde{f}) = \tilde{f} \circ \theta^* \circ i_Y$ for each $ f \in C_0(X)$.
\end{theorem}

\begin{proof}
\ref{QH1to1a}. 
Suppose  $\theta: C_0(X) \longrightarrow C_b(Y)$ is a bounded quasi-linear map
(resp., 
a conic quasi-linear map).
With $ \mathcal M$ and $w:Y \longrightarrow \mathcal M$ given by Theorem \ref{wEq},
consider an algebra homomorphism $H : C(\mathcal M) \longrightarrow C_b(Y)$ defined by $H(g) = g \circ w$. 
Let $ f \in C_0(X)$ (resp.,$ f \in C_0^+(X)$).  For each $y$  we have: 
$\theta(f) (y) =  \int f \, d w_y = \tilde{f} (w(y))$. 
Thus, $\theta(f) = \tilde{f} \circ w = H(\tilde{f}) = H(\Psi(f))$, so $ \theta = H \circ \Psi$ (resp., $ \theta = H \circ \Psi$ on $C_0^+(X)$).   \\
\ref{QH1to1b}. 
By Theorem \ref{ITqh} or Remark \ref{conicqlfExt}
 $\theta(f)(y) = \int_X f\, d(\theta^* \delta_y) = \tilde{f} (\theta^* \delta_y)$ 
for each $y$ and each  $ f \in C_0(X)$. Thus,
$ H(\tilde{f}) = \theta(f) = \tilde{f} \circ \theta^* \circ i_Y$.  \\
\ref{QH1to1c}. 
Let
$H : C(X^\sharp) \longrightarrow C_0(Y)$ (resp.,  $H : C(X^\flat) \longrightarrow C_0(Y)$) be an algebra homomorphism. 
By Theorem \ref{uThet11} $H(g) = g \circ u$ for some continuous proper map $u:Y \rightarrow X^\sharp$ (resp., $u:Y \rightarrow X^\flat$).
Then $ \theta = H \circ \Psi: C_0(X) \longrightarrow C_0(Y)$ is a quasi-homomorphism (resp., a conic quasi-homomorphism). 
Let $y \in Y$. For the (deficient) topological measure $u(y)$ and any $ f \in C_0^+(X)$ 
$$ \int_X f\, d(u(y)) = \tilde{f}(u(y)) = H(\tilde{f}) (y) = \theta(f) (y) = \int_X f \, d(\theta^* \delta_y).$$
By (\ref{EqByInt}) $u(y)=\theta^* \delta_y$, so $ u = \theta^* \circ i_Y$ 
(i.e. $u = w$ from Theorem \ref{wEq}).
Then $H(g) = g \circ u = g \circ \theta^* \circ i_Y$. 
\end{proof}

\begin{definition} \label{STMLC}
A signed topological measure on $X$ is a set function
$\mu:  \mathscr{C}(X) \cup \mathscr{O}(X)  \longrightarrow  [-\infty, \infty]$ that assumes at most one of $\infty, 
-\infty$ and satisfies the following conditions:
\begin{enumerate}[label=(STM\arabic*),ref=(STM\arabic*)]
\item \label{STM1} 
if $A,B, A \sqcup B \in \mathscr{K}(X) \cup \mathscr{O}(X) $ then
$
\mu(A\sqcup B)=\mu(A)+\mu(B);
$
\item \label{STM2}  
$
\mu(U)=\lim\{\mu(K):K \in \mathscr{K}(X), \  K \subseteq U\}
$ for $U\in\mathscr{O}(X)$;
\item \label{STM3}
$
\mu(F)=\lim\{\mu(U):U \in \mathscr{O}(X), \ F \subseteq U\}
$ for  $F \in \mathscr{C}(X)$.
\end{enumerate}
If in \ref{STM1}  $A,B \in \mathscr{K}(X)$ then $\mu$ is called a signed deficient topological measure. 
A signed deficient topological measure $ \mu$  is finite if $\| \mu \| =  \sup \{ | \mu(K)|  :  K \in \mathscr{K}(X)\} = \sup \{ | \mu(U)|  :  U \in \mathscr{O}(X) \} < \infty$.
We denote by $ \mathbf{SDTM}(X)$ the space of all finite signed deficient topological measures on $X$. 
\end{definition}

Signed deficient topological measures on locally compact spaces are studied in \cite{Butler:STMLC} and \cite{Butler:Decomp}.
Note that a finite signed deficient topological measure $ \mu$ can be written as $ \mu = \mu_1 - \mu_2$, where $ \mu_1, \mu_2$ are finite deficient topological measures, 
see \cite[Th. 40]{Butler:STMLC}.  

\begin{theorem} \label{mOnXstar}
Let $X$ be LC. Let $\lambda$ be a (deficient) topological measure with the corresponding quasi-integral $ \rho$. TFAE:
\begin{enumerate}[leftmargin=0.25in, label=(\roman*),ref=(\roman*)]
\item \label{mOnXstar1}
$ \lambda$ is a representable (deficient) topological measure.
\item \label{mOnXstar2}
There is a regular Borel probability measure $l$ on  $X^\sharp$ (resp., on $X^{\flat}$) such that for every  $ f \in C_0(X)$ (resp., $ f \in C_0^+(X)$)
$$\rho(f) =  \tilde f( \lambda )= \int_{X^{\flat}} \tilde{f} (\nu) \, dl(\nu). $$
\item \label{mOnXstar3}
There is a linear functional $L$ on $X^\sharp$ (resp., on $X^{\flat}$) of norm $1$ such that $ \rho = L \circ \Psi$.
\item \label{mOnXstar4}
$\lambda = l \circ \Lambda$.
\end{enumerate}
\end{theorem}

\begin{proof}
We shall prove the theorem for deficient topological measures, for topological measures the proof is similar.
\ref{mOnXstar1} $ \Longrightarrow$ \ref{mOnXstar2}:
We extend function $ \tilde f$ defined on  $ \mathbf{DTM}(X)$ to $\hat f$ on $ \mathbf{SDTM}(X)$  by $ \hat f (\mu) = \tilde f (\mu_1) - \tilde f (\mu_2)$ 
if $ \mu  = \mu_1 - \mu_2$, where $ \mu_1, \mu_2$ are finite deficient topological measures. Note that $ \hat f$ is well-defined: 
if  $\mu =   \mu_1 - \mu_2 = \nu_1 - \nu_2$ then $\tilde f (\mu_1 + \nu_2) = \tilde f (\mu_2 + \nu_1)$, so $ \hat f ( \mu_1 - \mu_2) = \hat f ( \nu_1 - \nu_2)$.
$\hat f$ is also linear on $ \mathbf{SDTM}(X)$.  If $ \mu  = \mu_1 - \mu_2 \neq 0$ then $ \mu_1  \neq \mu_2$, and there is $ f \in C_0(X)$ such that 
$ \tilde f (\mu_1) - \tilde f (\mu_2) \neq 0$. Thus, $ \{ \hat f: f \in C_0(X) \}$ separates points in  $ \mathbf{SDTM}(X)$.  
Let $ \mathcal F$ be the linear space generated by $ \{ \hat f: f \in C_0(X) \}$, and topology $ \tau = \sigma(\mathbf{SDTM}(X), \mathcal F)$. 
On $ \mathbf{DTM}(X)$, $\tau$  is the weak topology; also, the algebraic dual space of $ (\mathbf{SDTM}(X), \tau)$ is $\mathcal F$ (see \cite[Th. 3.10]{Rudin}.
By Lemma \ref{XdiezCom} $ X^{\flat}$ is compact, and so is the set of representable deficient topological measures 
(by \cite[Th. 2.4]{Butler:WkConv}). 
Then \cite[Th. 3.28]{Rudin} holds, where the vector-valued integral is with respect to the identity function;  
the vector-valued integral is according to \cite[Def. 3.26]{Rudin}, and it exists by \cite[Th. 3.27]{Rudin}. 
So $ \lambda$ is the vector-valued integral over $ X^{\flat}$ of the identity function with respect to some regular Borel probability  measure $l$.  
Applying the vector-valued integral with $\hat f =\tilde f$ to $ \lambda$ (as in   \cite[Def. 3.26]{Rudin}) gives \ref{mOnXstar2}. 
\ref{mOnXstar2} $ \Longrightarrow$ \ref{mOnXstar3}: Let $L$ be the linear functional on $C(X^{\flat})$ corresponding to $l$. Then
for any $f \in C_0^+(X)$
$$ (L \circ \Psi) (f) = L(\tilde f) = \int_{X^{\flat}} \tilde f( \nu) \, dl(\nu)  = \rho(f), $$
so $ \rho =  L \circ \Psi $. 
\ref{mOnXstar3} $ \Longrightarrow$ \ref{mOnXstar2}: clear.
\ref{mOnXstar2} $ \Longrightarrow$ \ref{mOnXstar4}: 
Note that  $l \circ \Lambda = \Lambda^* l$ is a deficient topological measure on $X$. 
Let $K \in \mathscr{K}(X)$, and let $f \in C_c(X), 1_K \le f \le 1$. For any 
$\nu \in K^{\flat}$ we have $ \int_X f\, d\nu = 1$, i.e. $\tilde f (\nu) = 1$. Thus, a continuous function 
$\tilde f = 1$ on $ K^{\flat}$. Then by Remark \ref{RemBRT}\ref{mrDTM}
\begin{align*}
 \lambda(K)  &= \inf \{ \rho(f) : f \in C_c(X), 1_K \le f \le 1  \}  = \inf \{ \int_{X^{\flat}}  \tilde f(\nu) \, dl(\nu), 1_{K^{\flat}} \le \tilde f \le 1  \}  \\
 & \ge l(K^{\flat}).
\end{align*}
Thus, $ \lambda \ge l \circ \Lambda$. 
Now let $U \in \mathscr{O}(X)$, and $f \in C_c(X), 0 \le f \le 1_U$. 
By \cite[Th. 49]{Butler:Integration},  $ \nu \notin U^{\flat} \Longrightarrow \nu(Coz (f)) = 0 \Longrightarrow \tilde f(\nu) = \int_X f\, d \nu = 0$, 
so $0 \le \tilde f \le  1_{U^{\flat}}$, and
by formula (\ref{muUprosche})
\begin{align*}
 \lambda(U) &= \sup\{ \rho(f): f \in C_c(X), 0 \le f \le 1_U\}  = \sup \{ \int_{X^{\flat}} \tilde f(\nu)\, dl(\nu), 0 \le  \tilde f \le  1_{U^{\flat}} \}  \\ 
 &\le l(U^{\flat}),
\end{align*}
showing $ \lambda \le l \circ \Lambda$. Thus,  $ \lambda = l \circ \Lambda$. 
\ref{mOnXstar4} $ \Longrightarrow$ \ref{mOnXstar1}: A probability measure can be approximated by a convex combination of point masses. 
\end{proof}

\begin{remark}
The measure $l$ in the above theorem may not be unique, as \cite[Ex. 6.1]{Aarnes:Pure} shows. When $X$ is compact, Theorem \ref{mOnXstar}
for topological measures is proved in \cite[Th. 4.1, 6.1]{Aarnes:Pure}.
When $X$ and $Y$ are compact, the part "in particular" in Theorem \ref{QH1to1} is \cite[Th. 25]{AarnesJohansenRustad}.
\end{remark}

In many theorems and examples we generalize
results that appeared for topological measures, image transformations, and quasi-homomorphisms
in the context (or close in spirit to the context) of compact spaces in  
\cite{Aarnes:Pure},
\cite{Aarnes:ITfirst},  \cite{AarnesTaraldsen}, \cite{AarnesGrubb}, \cite{AarnesJohansenRustad}, 
\cite{OrjanAlf:HomSimpleTM},  \cite{Pedersen}, 
\cite{AlfMultidimMedian}, \cite{AlfImTrans}, \cite{AlfMedian}, \cite{Taraldsen:RegQHinLC}, and 
\cite{Taraldsen:ITQMonLC}. Our generalizations occur in several directions: 
to locally compact spaces; to deficient topological measures; to d-image transformations; to (conic) quasi-linear maps.
Generalizing we use results from many sources. 
For instance, the majority of Theorem \ref{ITqh} 
is inspired by results from  \cite{Aarnes:ITfirst},   \cite{AarnesTaraldsen}, \cite{AarnesJohansenRustad}, \cite{Pedersen}, and \cite{Taraldsen:RegQHinLC}.
Many important results first appeared for topological measures and image transformations on compact spaces in \cite{Aarnes:ITfirst}, 
and then often also appear in one or more bibliography items in the above list. Such results are
Lemma \ref{ITsvva}\ref{ITsv1dop}, Example \ref{staIT}, Theorem \ref{ITw}, Example \ref{consQst}, 
Theorem \ref{ITfromUf}, definitions of $q^*$ and $ \theta^*$, Theorem \ref{q11eq},
parts \ref{thet1}-\ref{thet3}, \ref{thet5} of Theorem \ref{9eqThet}, Corollary \ref{ontoBij}, and Example \ref{exqxempt}. 
For compact spaces and image transformations, Theorem \ref{ITsolid}  is in \cite{AlfImTrans}, \cite{AarnesGrubb}, and \cite{Pedersen}; 
Example \ref{qhOnXY} is in \cite{AarnesGrubb}; and  parts \ref{laba2}-\ref{laba4}, \ref{laba7}, \ref{PsOponto} of Theorem \ref{LambCont} is in 
\cite{Aarnes:Pure}. 
In  the above mentioned studies different authors used different definitions of image transformations and quasi-homomorphisms, 
but our results show the equivalence of our definition to various definitions used before. 
  
In our next paper we will consider some applications of image transformations and quasi-homomorphisms on locally compact spaces. 


\subsection*{Acknowledgment}
The author would like to thank the Department of Mathematics at the University of California Santa Barbara for its supportive environment.

\end{document}